\documentclass[10pt]{article}
\usepackage{lmodern}
\usepackage{amsmath}
\usepackage[T1]{fontenc}
\usepackage[utf8]{inputenc}
\usepackage{authblk}
\usepackage{amsfonts}
\usepackage{graphicx}
\usepackage{rotating}
\usepackage{amssymb}
\usepackage[english]{babel}
\usepackage{color}
\usepackage{amsthm}
\usepackage{graphicx}
\usepackage{mathrsfs}
\usepackage{makecell}
\usepackage{microtype}
\usepackage{mathscinet}
\usepackage{array}
\usepackage{multirow}
\usepackage{enumerate}
\usepackage[cal=boondoxo,bb=ams]{mathalfa}
\usepackage{hyperref}
\hypersetup{hidelinks}

\newtheorem{defn}{Definition}[section]
\newtheorem{theorem}{Theorem}[section]
\newtheorem{prop}{Proposition}[section]
\newtheorem{lemma}{Lemma}[section]
\newtheorem{coro}{Corollary}[section]
\newtheorem{remark}{Remark}[section]
\newtheorem{exam}{Example}[section]

\newcommand{\ml}{\mathcal}
\newcommand{\mb}{\mathbb}

\title{On the critical regularity of nonlinearities for semilinear classical\\ wave equations}
\author[1]{Wenhui Chen\thanks{Wenhui Chen (wenhui.chen.math@gmail.com)}}
\affil[1]{School of Mathematics and Information Science, Guangzhou University, 510006 Guangzhou, China}
\author[2]{Michael Reissig\thanks{Michael Reissig (reissig@math.tu-freiberg.de)}}
\affil[2]{Institute of Applied Analysis, Faculty of Mathematics and Computer Science, Technical University Bergakademie Freiberg, 09596 Freiberg, Germany}
\date{}

\begin{document}

		\maketitle
		\begin{abstract}
			\medskip
			In this paper, we consider the Cauchy problem for   semilinear classical wave equations
			\begin{equation*}
			u_{tt}-\Delta u=|u|^{p_S(n)}\mu(|u|)
			\end{equation*}
			 with the Strauss exponent $p_S(n)$ and a modulus of continuity $\mu=\mu(\tau)$, which provides an additional regularity of nonlinearities in $u=0$ comparing with the power nonlinearity $|u|^{p_S(n)}$. We obtain a sharp condition on $\mu$ as a threshold between global (in time) existence of small data radial solutions by deriving polynomial-logarithmic type weighted $L^{\infty}_tL^{\infty}_r$ estimates, and blow-up of solutions in finite time even for small data by applying iteration methods with slicing procedure. These results imply the critical regularity of source nonlinearities for semilinear classical wave equations.
			\\
			
			
			\noindent\textbf{AMS Classification (2020)} 35L71, 35L05, 35B33, 35A01, 35B44
		\end{abstract}
\fontsize{12}{15}
\selectfont
\section{Introduction}\label{Section_Introduction}
$\ \ \ \ $In the last forty years, the Cauchy problem for  semilinear classical wave equations with power nonlinearity, namely,
\begin{align}\label{Semilinear-Wave}
\begin{cases}
	u_{tt}-\Delta u=|u|^{p},&x\in\mb{R}^n,\ t>0,\\
	u(0,x)=u_0(x),\ \ u_t(0,x)=u_1(x),&x\in\mb{R}^n,
\end{cases}	
\end{align}
with $p>1$, has been deeply studied by the mathematical community. For example, the questions on global (in time) existence of solutions, blow-up of solutions in finite time and sharp lifespan estimates of solutions were of interest.  In particular, the critical exponent for the semilinear Cauchy problem \eqref{Semilinear-Wave} is given by the so-called Strauss exponent $p_S(n)$, which was proposed by Walter A. Strauss in \cite{Strauss=1981}. Nowadays, the correctness of the Strauss exponent is well-known. The Strauss exponent $p_S(n)$ is the positive root of the quadratic equation
\begin{align}\label{Qud-Strauss}
	(n-1)p^2-(n+1)p-2=0
\end{align}
for $n\geqslant2$, that is,
\begin{align*}
p_S(n):=\frac{n+1+\sqrt{n^2+10n-7}}{2(n-1)} \ \ \mbox{when}\ \ n\geqslant2,
\end{align*}
 and we put $p_S(1):=+\infty$. On one hand, for blow-up results when $1<p\leqslant p_S(n)$, we refer interested readers to the classical papers \cite{John=1979,Kato=1980,Glassey=1981,Sideris=1984,Schaeffer=1985,Jiao-Zhou=2003,Yordanov-Zhang=2006,Zhou=2007} and the new proofs proposed in \cite{Wakasa-Yordanov=2019,Ikeda-Sobajima-Wakasa=2019}. On the other hand, concerning global (in time) existence results when $p>p_S(n)$, we refer to \cite{John=1979,Glassey=1981-GESDS,Lindblad-Sogge=1995,Georgiev-Lindblad=1997,Tataru=2001} and references therein. Summarizing these known results, in the scale of power nonlinearities $\{|u|^p\}_{p>1}$, the critical exponent $p=p_S(n)$ for semilinear classical wave equations \eqref{Semilinear-Wave} has been found, to be the threshold condition between global (in time) existence of solutions and blow-up of local (in time) solutions with small initial data.

Nevertheless, to determine the critical nonlinearity or the critical regularity of nonlinearities, it seems too rough to restrict the consideration of   semilinear wave equations \eqref{Semilinear-Wave} to the scale of power nonlinearities $\{|u|^p\}_{p>1}$. The question of the critical regularity of nonlinearities for semilinear classical wave equations is completely open as far as the authors know. For this reason, our contribution of this paper is to give an answer to this question for a class of modulus of continuity. Furthermore, we will suggest a candidate for the general critical nonlinearity via our derived results.

 In this manuscript, we consider the following Cauchy problem for semilinear classical wave equations with modulus of continuity in the nonlinearity:
\begin{align}\label{Semilinear-Wave-Modulus}
\begin{cases}
u_{tt}-\Delta u=|u|^{p_S(n)}\mu(|u|),&x\in\mb{R}^n,\ t>0,\\
u(0,x)=u_0(x),\ \ u_t(0,x)=u_1(x),&x\in\mb{R}^n,
\end{cases}
\end{align}
for $n\geqslant 2$ (due to $p_S(1)=+\infty$), where $p_S(n)$ stands for the Strauss exponent, and $\mu=\mu(\tau)$ is a modulus of continuity. To be specific, a function $\mu:[0,+\infty)\to[0,+\infty)$ is called a modulus of continuity, if $\mu$ is a continuous, concave and increasing function satisfying $\mu(0)=0$. The additional term of modulus of continuity provides an additional regularity of the nonlinear term in $u=0$ in the Cauchy problem \eqref{Semilinear-Wave-Modulus} comparing with the power nonlinearity $|u|^{p_S(n)}$. Note that the critical nonlinearity has been studied recently in semilinear classical damped wave equations \cite{Ebert-Girardi-Reissig=2020} and the corresponding weakly coupled systems \cite{Dao-Reissig=2021}. Nevertheless, due to the lack of crucial damping mechanisms, the study of the semilinear Cauchy problem \eqref{Semilinear-Wave-Modulus} is not a generalization of those of \cite{Ebert-Girardi-Reissig=2020,Dao-Reissig=2021}, e.g. the usual test function methods and the Matsumura type $L^p-L^q$ estimates do not work for our model \eqref{Semilinear-Wave-Modulus}.

The main purpose of this paper is to derive the critical regularity of nonlinearities for the semilinear Cauchy problem \eqref{Semilinear-Wave-Modulus}, namely, the threshold condition for modulus of continuity $\mu$. First of all, by applying iteration methods with slicing procedure (motivated by \cite{Agemi-Kurokawa-Takamura=2000,Wakasa-Yordanov=2019}) for a new weighted functional, which contains a local (in time) solution and a modulus of continuity, under some conditions of initial data, we will prove a blow-up result in Section \ref{Section-Blow-up} when
	\begin{align*}
	\lim_{\tau \to 0^+}\mu(\tau)\left(\log\frac{1}{\tau}\right)^{\frac{1}{p_S(n)}}\in[c_l,+\infty]
\end{align*}
with a suitably large constant $c_l\gg1$. Next, we will study the three dimensional Cauchy problem \eqref{Semilinear-Wave-Modulus} with modulus of continuity satisfying
	\begin{align*}
	 \lim_{\tau \to 0^+}\mu(\tau)\left(\log\frac{1}{\tau}\right)^{\frac{1}{p_S(3)}}=0
\end{align*}
 in the radial case. By developing polynomial-logarithmic type weighted $L^{\infty}_tL^{\infty}_r$ estimates via refined analysis in the $(t,r)$-plane, we will demonstrate global (in time) existence of small data radial solution in Section \ref{Section-GESDS}. The typical example is that for a modulus of continuity $\mu=\mu(\tau)$ with $\mu(0)=0$ which satisfies
\begin{align}\label{Int-example}
	\mu(\tau)=c_l\left(\log\frac{1}{\tau}\right)^{-\gamma} \ \ \mbox{with}\ \ c_l\gg1, \ \ \mbox{when}\ \ \tau\in(0,\tau_0].
\end{align}
Our results of this paper ensure that the critical regularity of nonlinearities $|u|^{p_S(3)}\mu(|u|)$ in semilinear classical wave equations with the modulus of continuity satisfying \eqref{Int-example} is described by the threshold
 $\gamma=\frac{1}{p_S(3)}$. Namely, global (in time) existence of solutions holds when $\gamma>\frac{1}{p_S(3)}$ and blow-up of solutions holds when $0<\gamma\leqslant\frac{1}{p_S(3)}$. Other examples will be shown in Section \ref{Section-Main-result}. To end this paper, we will give a conjecture for general conditions of the critical nonlinearity for the semilinear Cauchy problem \eqref{Semilinear-Wave-Modulus} as final remarks in Section \ref{Section-fINAL}.

\medskip
\noindent\textbf{Notation: } Firstly, $c$ and $C$ denote some positive constants, which may be changed
from line to line. We write $f\lesssim g$ if there exists a positive
constant $C$ such that $f\leqslant Cg$. The relation $f\simeq g$ holds if and only if $g\lesssim f\lesssim g$. Moreover, $B_R(0)$ denotes the ball around the origin with radius $R$. We denote $\langle y\rangle := 3+|y|$ for any $y\in\mb{R}$ throughout this manuscript.

\section{Main results}\label{Section-Main-result}
\setcounter{equation}{0}
$\ \ \ \ $Before stating our blow-up result, we firstly introduce the notion of energy solutions to the Cauchy problem \eqref{Semilinear-Wave-Modulus} that we are going to use later.
\begin{defn}\label{Defn-energy-solution}
	Let $u_0\in H^1$ and $u_1\in L^2$. We say that $u=u(t,x)$ is an energy solution to the semilinear Cauchy problem \eqref{Semilinear-Wave-Modulus} on $[0,T)$ if
	\begin{align*}
		u\in\ml{C}([0,T),H^1)\cap \ml{C}^1([0,T),L^2)\ \ \mbox{such that}\ \ |u|^{p_S(n)}\mu(|u|)\in L^1_{\mathrm{loc}}([0,T)\times\mb{R}^n)
	\end{align*}
	fulfills the next integral relation:
	\begin{align}\label{Eq.Defn.Energy.Solution.Crit.Case}
		&\int_{\mb{R}^n}u_t(t,x)\phi(t,x)\mathrm{d}x+\int_0^t\int_{\mb{R}^n}\big(\nabla u(s,x)\cdot\nabla \phi(s,x)-u_s(s,x)\phi_s(s,x)\big)\mathrm{d}x\mathrm{d}s\notag\\
		&\qquad=\int_{\mb{R}^n}u_1(x)\phi(0,x)\mathrm{d}x+\int_0^t\int_{\mb{R}^n}|u(s,x)|^{p_S(n)}\mu\big(|u(s,x)|\big)\phi(s,x)\mathrm{d}x\mathrm{d}s
	\end{align}
	for any $\phi\in\ml{C}_0^{\infty}([0,T)\times\mb{R}^n)$ and any $t\in[0,T)$.
\end{defn}
\begin{theorem}\label{Thm-Blow-up}
	Let us consider a modulus of continuity $\mu=\mu(\tau)$ with $\mu(0)=0$ satisfying
	\begin{align}\label{Assumption-Blow-up}
\lim_{\tau \to 0^+}\mu(\tau)\left(\log\frac{1}{\tau}\right)^{\frac{1}{p_S(n)}}=:C_{\mathrm{Str}}\in[c_l,+\infty]
	\end{align}
with a suitably large constant $c_l\gg1$.
We assume that the function $g:\tau\in\mb{R}\to g(\tau):=\tau[\mu(|\tau|)]^{\frac{1}{p_S(n)}}$ is convex on $\mb{R}$.
 Let $u_0\in H^1$ and $u_1\in L^2$ be non-negative, non-trivial and compactly supported functions with supports contained in $B_R(0)$ for some $R>0$. Let
\begin{align*}
	u\in\ml{C}([0,T),H^1)\cap \ml{C}^1([0,T),L^2)\ \ \mbox{such that}\ \ |u|^{p_S(n)}\mu(|u|)\in L^1_{\mathrm{loc}}([0,T)\times\mb{R}^n)
\end{align*}
be an energy solution to the semilinear Cauchy problem \eqref{Semilinear-Wave-Modulus} on $[0,T)$ for $n\geqslant 2$ according to Definition \ref{Defn-energy-solution}. Then, the energy solution $u$ blows up in finite time.
\end{theorem}
\begin{exam}\label{Example-blow-up}
The hypothesis \eqref{Assumption-Blow-up} and the supposed property for the function $g=g(\tau)$ of Theorem \ref{Thm-Blow-up} hold for the following functions $\mu=\mu(\tau)$ on a small interval $[0,\tau_0]$ with $0<\tau_0\ll 1$:
\begin{itemize}
	\item $\mu(0)=0$ and $\mu(\tau)=(\log\frac{1}{\tau})^{-\gamma}$ with $0<\gamma<\frac{1}{p_S(n)}$;
	\item $\mu(0)=0$ and $\mu(\tau)=c_l(\log\frac{1}{\tau})^{-\frac{1}{p_S(n)}}$ with $c_l\gg1$;
	\item $\mu(0)=0$ and $\mu(\tau)=(\log\frac{1}{\tau})^{-\frac{1}{p_S(n)}}(\log^k\frac{1}{\tau})^{\gamma}$ with $\gamma>0$ and $k\geqslant 2$, here $\log^k$
denotes the iterated logarithm ($k$ times application).
\end{itemize}
 Note that the modulus of continuity in the last cases can be continued to $\tau\in[0,+\infty)$ in such a way that $\mu=\mu(\tau)$ is a continuous, concave and increasing function, for example, a smooth and concave continuation function with $\mu(0)=0$ such that
\begin{align*}
\mu(\tau)=\begin{cases}
	c_l(\log\frac{1}{\tau})^{-\gamma}&\mbox{when}\ \ \tau\in(0,\frac{1}{3}],\\
	\mbox{strictly increasing}&\mbox{when}\ \ \tau\in[\frac{1}{3},3],\\
	c_l(\log \tau)^{\gamma}&\mbox{when}\ \ \tau\in[3,+\infty),
\end{cases}
\end{align*}
with   a suitably large constant $c_l\gg1$ and  $0<\gamma\leqslant\frac{1}{p_S(n)}$. A counterexample for the condition \eqref{Assumption-Blow-up} in Theorem \ref{Thm-Blow-up} is $\mu(\tau)=\tau^{\nu}$ with $\nu>0$. This is not surprising due to the global (in time) existence results \cite{John=1979,Glassey=1981-GESDS,Lindblad-Sogge=1995,Georgiev-Lindblad=1997,Tataru=2001} for the  semilinear classical wave equations \eqref{Semilinear-Wave} with power nonlinearity $|u|^{p_S(n)+\nu}$.
\end{exam}
%\begin{exam}
% the hypothesis \eqref{Assumption-Blow-up} and the supposed properties for the function $g=g(\tau)$ of Theorem \ref{Thm-Blow-up} hold for the following function $\mu=\mu(\tau)$ on a small interval $[0,\tau_0]$ with $0<\tau_0\ll 1$:
%
%\end{exam}
%\begin{remark}
%Let us discuss the assumption in Theorem \ref{Thm-Blow-up} for the function $g=g(\tau)$. In the case of smooth $\mu$ in a small right-sided neighborhood of $\tau=0$, this convex assumption can be replaced by the conditions
%\begin{align}\label{Cond-Rem}
%	\tau\mu'(\tau)=o(\mu(\tau))\ \ \mbox{and}\ \ \tau\mu''(\tau)=o(\mu'(\tau))\ \ \mbox{when}\ \ \tau\in(0,\tau_0].
%\end{align}
%Indeed, it is sufficient to verify that on a small interval $(0,\tau_0]$
%\begin{align*}
%g''(\tau)=\frac{\mu'(\tau)}{p_S(n)}[\mu(\tau)]^{\frac{1}{p_S(n)}-1}\left[2-\frac{p_S(n)-1}{p_S(n)}\frac{\tau\mu'(\tau)}{\mu(\tau)}+\frac{\tau\mu''(\tau)}{\mu'(\tau)}\right]\geqslant0.
%\end{align*}
%These conditions \eqref{Cond-Rem} are satisfied in our examples. Outside this small interval $(0,\tau_0]$, we can choose convex continuation of $g(\tau)$ function.
%\end{remark}
\begin{remark}\label{Remark2.1}
Concerning the semilinear wave equation \eqref{Semilinear-Wave} with the critical exponent $p=p_S(n)$, by taking the additional term of modulus of continuity $\mu(|u|)$  fulfilling \eqref{Assumption-Blow-up} in the nonlinearity, Theorem \ref{Thm-Blow-up} shows that the energy solutions still blow up in finite time.
\end{remark}

To indicate the sharpness of the condition \eqref{Assumption-Blow-up}, we next study the three dimensional semilinear Cauchy problem \eqref{Semilinear-Wave-Modulus} with a modulus of continuity satisfying \eqref{GESDS-Condition}. Before showing our result, taking $r=|x|$, let us introduce a definition of radial solutions to our aim model in three dimensions, namely,
\begin{align}\label{G0}
	\begin{cases}
		\displaystyle{u_{tt}-u_{rr}-\frac{2}{r}u_r=|u|^{p_S(3)}\mu(|u|),}&r>0,\ t>0,\\
		u(0,r)=u_0(r),\ \ u_t(0,r)=u_1(r),&r>0.
	\end{cases}
\end{align}
\begin{defn}\label{Defn-mild-solution}
	The function $u=u(t,r)$ is called a global (in time) mild solution to the semilinear Cauchy problem \eqref{G0} if $u\in\ml{C}([0,+\infty)\times\mb{R}_+)$ carrying its initial data, and satisfying the following integral equality:
	\begin{align*}
	u(t,r)=\ml{E}_0(t,r)\ast_{(r)} u_0(r)+\ml{E}_1(t,r)\ast_{(r)} u_1(r)+\int_0^t\ml{E}_1(t-s,r)\ast_{(r)}\big[|u(s,r)|^{p_S(3)}\mu\big(|u(s,r)|\big)\big]\mathrm{d}s.
	\end{align*}
In the above, $\ml{E}_0=\ml{E}_0(t,r)$ and $\ml{E}_1=\ml{E}_1(t,r)$ are the fundamental solutions to the corresponding linear Cauchy problem to \eqref{G0}
with vanishing right-hand side.
\end{defn}
\noindent We turn to the global (in time) existence of radial solutions in the subsequent theorem.
\begin{theorem}\label{Thm-GESDS}
Let us consider a modulus of continuity $\mu=\mu(\tau)$ with $\mu(0)=0$ satisfying
\begin{align}
\lim_{\tau \to 0^+}\mu(\tau)\left(\log\frac{1}{\tau}\right)^{\frac{1}{p_S(3)}}=0,\label{GESDS-Condition}
\end{align}
furthermore,
\begin{align}
\mu(\tau)\left(\log\frac{1}{\tau}\right)^{\frac{1}{p_S(3)}}\lesssim \left(\log\log\frac{1}{\tau}\right)^{-1}\ \ \mbox{when}\ \ \tau\in(0,\tau_0].\label{Speical-mu}
\end{align}
Let $\bar{u}_0\in\ml{C}_0^2$ and $\bar{u}_1\in\ml{C}_0^1$ be radial. Then, there exists $0<\varepsilon_0\ll 1$ such that for any $\varepsilon\in(0,\varepsilon_0)$, if $u_0=\varepsilon\bar{u}_0$ and $u_1=\varepsilon\bar{u}_1$, then the semilinear Cauchy problem \eqref{Semilinear-Wave-Modulus}  for $n=3$ admits a uniquely determined global (in time) small data radial solution in the sense of Definition \ref{Defn-mild-solution} such that $u\in\ml{C}([0,+\infty)\times\mb{R}^3)$.
\end{theorem}
\begin{remark}
Since the assumption \eqref{Speical-mu} implies \eqref{GESDS-Condition} as $\tau\to 0^+$, one may drop the condition \eqref{GESDS-Condition} directly. Nevertheless, to emphasize the importance of the essential condition \eqref{GESDS-Condition} for the  global (in time) existence result, we retain this condition. We conjecture that the logarithmic type decay condition \eqref{Speical-mu} is a technical restriction.
\end{remark}
\begin{exam}\label{Example-GESDS}
The hypotheses \eqref{GESDS-Condition} and \eqref{Speical-mu} hold for the following functions $\mu=\mu(\tau)$ on a small interval $[0,\tau_0]$ with $0<\tau_0\ll 1$:
\begin{itemize}
	\item $\mu(\tau)=\tau^{\gamma}$ with $\gamma\in(0,1]$;
	\item $\mu(\tau)=[\log(1+\tau)]^{\gamma}$ with $\gamma\in(0,1]$;
	\item $\mu(0)=0$ and $\mu(\tau)=(\log\frac{1}{\tau})^{-\gamma}$ with $\gamma>\frac{1}{p_S(3)}$;
	\item $\mu(0)=0$ and $\mu(\tau)=(\log\frac{1}{\tau})^{-\frac{1}{p_S(3)}}(\log\log\frac{1}{\tau})^{\gamma}$ with $\gamma\leqslant -1$;
	\item $\mu(0)=0$ and $\mu(\tau)=(\log\frac{1}{\tau})^{-\frac{1}{p_S(3)}}(\log\log\frac{1}{\tau})^{-1}(\log^k\frac{1}{\tau})^{\gamma}$ with $\gamma<0$ and $k\geqslant 3$.
\end{itemize}
\end{exam}

\begin{remark} \label{Remark2.3}
By assuming additionally decay properties for initial data with respect to the radial behavior, we also can derive some pointwise decay estimates for the global (in time) radial solutions. More details will be given in Corollary \ref{Coro-Decay} and in our proof in Section \ref{Section-GESDS}.
\end{remark}
\begin{remark} \label{Remark2.4}
The key tool to prove Theorem \ref{Thm-GESDS} is to derive polynomial-logarithmic type weighted $L^{\infty}_tL^{\infty}_r$ estimates. Concerning higher dimensional cases, one may recall more general representations of radial solutions to the linear wave equation associated with polynomial type weighted $L^{\infty}_tL^{\infty}_r$ estimates (see \cite{Kubo=1994,Kubo-Kubota=1995} for odd dimensions and \cite{Kubo-Kubota=1998} for even dimensions). Furthermore, by setting suitable logarithmic factors to be the additional part of weighted functions, one may derive some weighted $L^{\infty}_tL^{\infty}_r$ estimates to get a global (in time) existence result for higher dimensions $n$, nevertheless, this purpose is beyond the scope of this manuscript.
\end{remark}

\begin{remark} \label{Remark2.2}
	Let us summarize the given results in Theorems \ref{Thm-Blow-up} and \ref{Thm-GESDS}. We recall the typical modulus of continuity proposed in Examples \ref{Example-blow-up} and \ref{Example-GESDS}. In the consideration of semilinear wave equations \eqref{Semilinear-Wave-Modulus}  for $n=3$ with the modulus of continuity satisfying \eqref{Int-example}, we may conclude that the critical regularity of nonlinearities is described by the threshold $\gamma=\frac{1}{p_S(3)}$. This is one of the main contributions of this paper and it answers the open question proposed in the introduction.
\end{remark}

\begin{remark}\label{Rem-GESDS-CONjecture}
Motivated by the global (in time) existence condition \eqref{GESDS-Condition} as well as the blow-up condition \eqref{Assumption-Blow-up}, one may introduce the following possible quantity:
\begin{align}\label{Conjecture-Chen-Reissig}
	0\leqslant C_{\mathrm{Str}}:=\lim_{\tau \to 0^+}\mu(\tau)\left(\log\frac{1}{\tau}\right)^{\frac{1}{p_S(n)}},
\end{align}
to describe the critical regularity of nonlinearities for semilinear wave equations \eqref{Semilinear-Wave-Modulus}. The blow-up phenomenon occurs when $C_{\mathrm{Str}}\in[c_l,+\infty]$ in Theorem \ref{Thm-Blow-up} and the global (in time) existence result holds when $C_{\mathrm{Str}}=0$ in Theorem \ref{Thm-GESDS}. Explanations more in detail will be provided in Section \ref{Section-fINAL}.
\end{remark}

\section{Blow-up of energy solutions}\label{Section-Blow-up}
\setcounter{equation}{0}
$\ \ \ \ $This section is organized as follows. In Subsection \ref{Sub-3.1}, we will introduce a test function, and derive sharp estimates for it in $L^1(B_{R+t}(0))$. Then, thanks to some estimates for auxiliary functions, the iteration frame and lower bound estimates for a time-dependent functional will be established in Subsections \ref{Sub-3.2} and \ref{Sub-3.3}, respectively. Finally, in Subsection \ref{Sub-3.4}, we will demonstrate the lower bound of this functional blows up in finite time by using iteration methods with slicing procedure.
\subsection{Preliminaries and auxiliary functions}\label{Sub-3.1}
$\ \ \ \ $Let us set a non-negative parameter
\begin{align}\label{parameter-r}
q:=\frac{n-1}{2}-\frac{1}{p_S(n)} \ \ \mbox{for}\ \ n\geqslant 2.
\end{align}
 Next, we recall the following pair of auxiliary functions from \cite{Wakasa-Yordanov=2019}:
\begin{align}
	\xi_q(t,x) & :=  \int_0^{\lambda_0} \mathrm{e}^{-\lambda(R+t)} \cosh (\lambda t)  \Phi(\lambda x)  \lambda^q  \mathrm{d}\lambda,\label{def-xi}\\
	\eta_q(t,s,x) & :=  \int_0^{\lambda_0} \mathrm{e}^{-\lambda(R+t)} \frac{\sinh (\lambda (t-s))}{\lambda(t-s)} \Phi(\lambda x) \lambda^q  \mathrm{d}\lambda, \label{def-eta}
\end{align} where $\lambda_0$ is a fixed positive parameter and the test function $\Phi=\Phi(x)$ defined by \begin{align*}
\Phi: x \in \mathbb{R}^n \to	\Phi(x) :=
	\int_{\mathbb{S}^{n-1}} \mathrm{e}^{x\cdot \omega}\mathrm{d} \sigma_\omega\ \ \mbox{for}\ \ n\geqslant 2,
\end{align*}
was introduced by \cite{Yordanov-Zhang=2006}.  The test function $\Phi$ is positive, smooth, and satisfies $\Delta \Phi =\Phi$ with
\begin{align}
	 \Phi (x) \simeq |x|^{-\frac{n-1}{2}}  \mathrm{e}^{|x|} \quad \mbox{as} \quad |x|\to +\infty.\label{Asymp}
\end{align}
By introducing the function with separate variables
\begin{align*}
\Psi(t,x):=\mathrm{e}^{-t}\,\Phi(x),
\end{align*}
 it is the solution to the free wave equation $\Psi_{tt}-\Delta\Psi=0$ and has the next property.
\begin{lemma}\label{Lemma-Optimal-Est}The test function fulfills the sharp estimates
\begin{align*}
\int_{B_{R+t}(0)}\Psi(t,x)\mathrm{d}x\simeq(R+t)^{\frac{n-1}{2}}
\end{align*}
for any $t\geqslant0$ and $n\geqslant 2$.
\end{lemma}
\begin{proof}
By using integration by parts, we arrive at
\begin{align*}
\mathrm{e}^{-t}\int_0^{R+t}\zeta^{\frac{n-1}{2}}\mathrm{e}^\zeta\mathrm{d}\zeta&=(R+t)^{\frac{n-1}{2}}\mathrm{e}^R-\frac{n-1}{2}\mathrm{e}^{-t}\int_0^{R+t}\zeta^{\frac{n-3}{2}}\mathrm{e}^\zeta\mathrm{d}\zeta\lesssim (R+t)^{\frac{n-1}{2}}.
\end{align*}
Shrinking the domain of integration to $[t,R+t]$, one notices
\begin{align*}
	\mathrm{e}^{-t}\int_0^{R+t}\zeta^{\frac{n-1}{2}}\mathrm{e}^\zeta\mathrm{d}\zeta\geqslant \int_t^{R+t}\zeta^{\frac{n-1}{2}}\mathrm{d}\zeta=\frac{2}{n+1}\left((R+t)^{\frac{n+1}{2}}-t^{\frac{n+1}{2}} \right)\gtrsim(R+t)^{\frac{n-1}{2}}.
\end{align*}
Therefore, the previous sharp estimates imply
\begin{align*}
	 \int_{B_{R+t}(0)}\Psi(t,x)\mathrm{d}x\simeq\int_{B_{R+t}(0)}|x|^{-\frac{n-1}{2}}\mathrm{e}^{|x|-t}\mathrm{d}x\simeq\mathrm{e}^{-t}\int_0^{R+t}\zeta^{\frac{n-1}{2}}\mathrm{e}^\zeta\mathrm{d}\zeta\simeq(R+t)^{\frac{n-1}{2}}
\end{align*}
because of \eqref{Asymp}. The proof is completed.
\end{proof}

Additionally, some useful estimates of $\xi_q$ and $\eta_q$ are stated in the following lemma, whose proof can be found in \cite[Lemma 3.1]{Wakasa-Yordanov=2019}. Note that our setting of $q$ fulfills all assumptions in Lemma \ref{lemma eta and xi estimates}. Moreover, we recall the notation $\langle y \rangle =3+|y|$.
\begin{lemma} \label{lemma eta and xi estimates} There exists $\lambda_0>0$ such that the following properties hold for $n\geqslant 2$:
	\begin{itemize}
		\item[\rm{(i)}] if $q>-1$, $|x|\leqslant R$ and $t\geqslant 0$, then
		\begin{align*}
			\xi_q(t,x) & \geqslant A_0, \\
			\eta_q(t,0,x) & \geqslant B_0 \langle t\rangle^{-1};
		\end{align*}
		\item[\rm{(ii)}] if $q>-1$, $|x|\leqslant R+s$ and $t>s\geqslant 0$, then
		\begin{align*}
			\eta_q(t,s,x) & \geqslant B_1 \langle t\rangle^{-1} \langle s\rangle^{-q};
		\end{align*}
		\item[\rm{(iii)}] if $q>\frac{n-3}{2}$, $|x|\leqslant R+t$ and $t> 0$, then
		\begin{align*}
			\eta_q(t,t,x) & \leqslant B_2 \langle t\rangle^{-\frac{n-1}{2}} \langle t-|x| \rangle^{\frac{n-3}{2}-q}.
		\end{align*}
	\end{itemize}
	Here, $A_0$ and $B_k$, with $k=0,1,2$, are positive constants depending only on $\lambda_0$, $q$ and $R$.
\end{lemma}

To end this subsection, we include the following generalized version of Jensen's inequality \cite{Pick-Kufner-John-Fucik=2013}, whose proof also has been shown in \cite[Lemma 8]{Ebert-Girardi-Reissig=2020}.
\begin{lemma}\label{Lem-EGR}
	Let $g=g(\tau)$ be a convex function on $\mb{R}$. Let $\alpha=\alpha(x)$ be defined and non-negative almost everywhere on $\Omega$, such that $\alpha$ is positive in a set of positive measure. Then, it holds
	\begin{align*}
		 g\left(\frac{\int_{\Omega}v(x)\alpha(x)\mathrm{d}x}{\int_{\Omega}\alpha(x)\mathrm{d}x}\right)\leqslant\frac{\int_{\Omega}g(v(x))\alpha(x)\mathrm{d}x}{\int_{\Omega}\alpha(x)\mathrm{d}x}
	\end{align*}
	for all non-negative functions $v=v(x)$ provided that all the integral terms are meaningful.
\end{lemma}
\subsection{Construction of an iteration frame}\label{Sub-3.2}
$\ \ \ \ $In order to prove Theorem \ref{Thm-Blow-up}, we are going to use an iteration argument to derive lower bound estimates for the weighted space average of a local (in time) solution containing modulus of continuity. For this reason, we first derive a nonlinear integral inequality to get an iteration frame.
\begin{prop} \label{Prop-lower-bounds-critical-case} Let   $u_0\in H^1$ and $u_1\in  L^2$ be non-negative, non-trivial and compactly supported functions with supports contained in $B_R(0)$ for some $R>0$. Let $u$ be an energy solution to the semilinear Cauchy problem \eqref{Semilinear-Wave-Modulus} on $[0,T)$  according to Definition \ref{Defn-energy-solution}. Then, the following integral identity holds:
	\begin{align}
		\int_{\mathbb{R}^n}u(t,x)  \eta_{q}(t,t,x)  \mathrm{d}x &=  \int_{\mathbb{R}^n} u_0(x) \xi_{q}(t,x)  \mathrm{d}x +  t \int_{\mathbb{R}^n}u_1(x) \eta_{q}(t,0,x)  \mathrm{d}x\notag\\ &  \quad + \int_0^t(t-s)\int_{\mb{R}^n}|u(s,x)|^{p_S(n)}\mu\big(|u(s,x)|\big)\eta_q(t,s,x)\mathrm{d}x\mathrm{d}s \label{fund ineq G}
	\end{align} for any $t\in (0,T)$, where $\xi_q$ and $\eta_q$ are defined in \eqref{def-xi} and \eqref{def-eta}, respectively.
\end{prop}
\begin{proof}
	From finite propagation speed for solutions of wave equations,  $u(t,\cdot)$ has compact support contained in $B_{R+t}(0)$ for any $t\geqslant 0$. Therefore, we may employ \eqref{Eq.Defn.Energy.Solution.Crit.Case} for a non-compactly supported test function.
	We now define the test function
	\begin{align*}
	\psi=\psi(s,x):=y(t,s;\lambda)\Phi(\lambda x)\ \ \mbox{with}\ \ y(t,s;\lambda):=\frac{ \sinh (\lambda(t-s))}{\lambda}.
	\end{align*}
	As $\Phi$ is an eigenfunction of the Laplacian and $y(t,s;\lambda)$ solves $(\partial_s^2 -\lambda^2)y(t,s;\lambda)=0$  with the end-points $y(t,t;\lambda)=0$ and $ y_s(t,t;\lambda)=-1$, the function $\psi$ solves the free wave equation $\psi_{ss}-\Delta \psi =0 $
	and satisfies
	\begin{align*}
		&\,\, \psi(t,x)= 0 , \qquad \quad \, \qquad \psi(0,x)= \lambda^{-1} \sinh (\lambda t) \Phi(\lambda x), \\
		& \psi_s(t,x)= - \Phi(\lambda x) , \qquad \psi_{s}(0,x)= -\cosh (\lambda t) \Phi(\lambda x).
	\end{align*}
	Applying the test function $\psi$ in \eqref{Eq.Defn.Energy.Solution.Crit.Case} with an integration by parts once more, we may derive
	\begin{align*}
		\int_{\mathbb{R}^n}  u(t,x)  \Phi(\lambda x)  \mathrm{d}x  &=   \cosh (\lambda t) \int_{\mathbb{R}^n}  u_0(x)  \Phi(\lambda x)  \mathrm{d}x +  t\frac{\sinh (\lambda t)}{\lambda t}  \int_{\mathbb{R}^n}   u_1(x)  \Phi(\lambda x)  \mathrm{d}x \\ & \quad + \int_0^t (t-s)  \frac{\sinh (\lambda(t-s))}{\lambda(t-s)} \int_{\mathbb{R}^n}|u(s,x)|^{p_S(n)}\mu\big(|u(s,x)|\big) \Phi(\lambda x)  \mathrm{d}x  \mathrm{d}s.
	\end{align*}
%\textcolor{green}{check it again, I got another relation} \color{red} It is correct: indeed, we have
%\begin{align*}
%	&\int_{\mb{R}^n} u_t(t,x)\underbrace{\psi(t,x)}_{=0}\mathrm{d}x+\int_0^t\int_{\mb{R}^n}u(s,x)\underbrace{\big(\psi_{ss}(s,x)-\Delta\psi(s,x)\big)}_{=0}\mathrm{d}x\mathrm{d}s\\
%	&-\int_{\mb{R}^n}u(t,x)\underbrace{\psi_s(t,x)}_{=-\Phi(\lambda x)}\mathrm{d}x+\int_{\mb{R}^n}u_0(x)\underbrace{\psi_s(0,x)}_{=-\cosh(\lambda t)\Phi(\lambda x)}\mathrm{d}x\\
%	&\qquad=\int_{\mb{R}^n}u_1(x)\underbrace{\psi(0,x)}_{=\lambda^{-1}\sinh(\lambda t)\Phi(\lambda x)}\mathrm{d}x+\int_0^t\int_{\mb{R}^n}|u(s,x)|^{p_S(n)}\mu\big(|u(s,x)|\big)\underbrace{\psi(s,x)}_{=\lambda^{-1}\sinh(\lambda (t-s))\Phi(\lambda x)}\mathrm{d}x\mathrm{d}s.
%\end{align*}
Multiplying both sides of the last equality by $\mathrm{e}^{-\lambda(R+t)}\lambda^{q}$, integrating the resultant with respect to $\lambda$ over $[0,\lambda_0]$ and applying Tonelli's theorem, we complete the derivation of  \eqref{fund ineq G}. 	
\end{proof}

Hereafter until the end of this section, we shall assume that $u_0,u_1$ satisfy the assumptions from  Theorem \ref{Thm-Blow-up}. Let $u$ be an energy solution to the semilinear Cauchy problem \eqref{Semilinear-Wave-Modulus} on $[0,T)$. Inspired by the modulus of continuity in its nonlinearity, let us introduce the time-dependent functional
\begin{align*}
	\ml{U}(t):= \int_{\mathbb{R}^n} u(t,x)\big[\mu\big(|u(t,x)|\big)\big]^{\frac{1}{p_S(n)}} \eta_{q}(t,t,x) \mathrm{d}x
\end{align*}
with the parameter $q$ defined in \eqref{parameter-r}. Moreover, it follows immediately the non-negativity of the functional  for any $t\geqslant0$ by
\begin{align}\label{Non-negativity-U}
	\ml{U}(t)\geqslant\int_{\mathbb{R}^n} u(t,x) \eta_{q}(t,t,x) \mathrm{d}x\inf\limits_{x\in\mb{R}^n}\big[\mu\big(|u(t,x)|\big)\big]^{\frac{1}{p_S(n)}}\geqslant0,
\end{align}
where we employed as a direct consequence $\int_{\mathbb{R}^n} u(t,x) \eta_{q}(t,t,x) \mathrm{d}x\geqslant0$ from \eqref{fund ineq G}, with the help of non-negative data and non-negative nonlinearity.	

 A further step is to derive some estimates involving $\ml{U}(t)$ both in the left- and right-hand sides, which will establish an iteration frame. According to \eqref{fund ineq G} and non-negativity of initial data, we may claim
\begin{align}\label{Eq-C-01}
	\int_{\mathbb{R}^n}u(t,x)  \eta_{q}(t,t,x)  \mathrm{d}x\geqslant \int_0^t(t-s)\int_{\mb{R}^n}|u(s,x)|^{p_S(n)}\mu\big(|u(s,x)|\big)\eta_q(t,s,x)\mathrm{d}x\mathrm{d}s.
\end{align}
Using H\"older's inequality, we arrive at
\begin{align*}
	\mathcal{U}(s) \leqslant \left(\int_{\mathbb{R}^n}|u(s,x)|^{p_S(n)}\mu\big(|u(s,x)|\big) \eta_{q}(t,s,x) \mathrm{d}x\right)^{\frac{1}{p_S(n)}} \left(\int_{B_{R+s}}\frac{[\eta_{q}(s,s,x)]^{p'_S(n)}}{[\eta_{q}(t,s,x)]^{\frac{p'_S(n)}{p_S(n)}}} \mathrm{d}x\right)^{\frac{1}{p'_S(n)}}.
\end{align*}
Remark that $p_S'(n)$ denotes H\"older's conjugate of $p_S(n)$.
With the aid of the properties (ii) and (iii) in Lemma \ref{lemma eta and xi estimates} (both $q >\frac{n-3}{2}$ and $q>-1$ are always fulfilled),  we obtain
\begin{align*}
	\int_{B_{R+s}}\frac{[\eta_{q}(s,s,x)]^{p'_S(n)}}{[\eta_{q}(t,s,x)]^{\frac{p'_S(n)}{p_S(n)}}} \mathrm{d}x & \lesssim \langle t\rangle^{\frac{p'_S(n)}{p_S(n)}}\langle s\rangle^{\frac{p'_S(n)}{p_S(n)}q-\frac{n-1}{2}p'_S(n)}\int_{B_{R+s}}\langle s-|x|\rangle^{(\frac{n-3}{2}-q)p'_S(n)}  \mathrm{d}x \\
	& \lesssim  \langle t\rangle^{\frac{p'_S(n)}{p_S(n)}}\langle s\rangle^{\frac{q}{p_S(n)-1}-\frac{n-1}{2}p'_S(n)}\int_{B_{R+s}}\langle s-|x|\rangle^{-1}  \mathrm{d}x \\
	& \lesssim  \langle t\rangle^{\frac{p'_S(n)}{p_S(n)}}\langle s\rangle^{\frac{p'_S(n)}{p_S(n)}}\log\langle s\rangle,
\end{align*}
due to our choice of $q$ in \eqref{parameter-r} and
\begin{align*}
\frac{q}{p_S(n)-1}-\frac{n-1}{2}p'_S(n)+n-1&=\frac{p'_S(n)}{p_S(n)}\left(\frac{n-1}{2}-\frac{1}{p_S(n)}-\frac{n-1}{2}p_S(n)+(n-1)\big(p_S(n)-1\big)\right)\\
&=\frac{p'_S(n)}{p_S(n)}\left[\frac{1}{p_S(n)}\left(\frac{n-1}{2}p^2_S(n)-\frac{n+1}{2}p_S(n)-1\right)+1\right]=\frac{p'_S(n)}{p_S(n)}.
\end{align*}
Note that $\log\langle s\rangle\geqslant\log3>0$. Plugging the previous estimates in \eqref{Eq-C-01}, it leads to
\begin{align}\label{Cn01}
	\int_{\mathbb{R}^n}u(t,x)  \eta_{q}(t,t,x)  \mathrm{d}x &\gtrsim \int_0^t (t-s)[\ml{U}(s)]^{p_S(n)}\left(\int_{B_{R+s}}\frac{[\eta_{q}(s,s,x)]^{p'_S(n)}}{[\eta_{q}(t,s,x)]^{\frac{p'_S(n)}{p_S(n)}}} \mathrm{d}x\right)^{-\frac{p_S(n)}{p'_S(n)}}  \mathrm{d}s   \notag\\
	& \gtrsim \langle t\rangle^{-1}\int_0^t (t-s) \langle s\rangle^{-1} \frac{[\mathcal{U}(s)]^{p_S(n)}}{\left(\log\langle s\rangle\right)^{p_S(n)-1}}    \mathrm{d}s.
\end{align}
Moreover, thanks to the support condition of $u(t,\cdot)$, let us apply Lemma \ref{Lem-EGR} with $\Omega=B_{R+t}(0)$, $\alpha=\eta_q(t,t,x)$, $v=u(t,x)$ and the convex function $g=g(\tau)=\tau[\mu(|\tau|)]^{\frac{1}{p_S(n)}}$ from our assumption in Theorem \ref{Thm-Blow-up} to deduce
\begin{align*}
	 g\left(\frac{\int_{B_{R+t}(0)}u(t,x)\eta_q(t,t,x)\mathrm{d}x}{\int_{B_{R+t}(0)}\eta_q(t,t,x)\mathrm{d}x}\right)\leqslant\frac{\ml{U}(t)}{\int_{B_{R+t}(0)}\eta_q(t,t,x)\mathrm{d}x},
\end{align*}
in other words,
\begin{align}\label{C2}
	\int_{B_{R+t}(0)}u(t,x)\eta_q(t,t,x)\mathrm{d}x\leqslant \int_{B_{R+t}(0)}\eta_q(t,t,x)\mathrm{d}x \  g^{-1}\left(\frac{\ml{U}(t)}{\int_{B_{R+t}(0)}\eta_q(t,t,x)\mathrm{d}x} \right).
\end{align}
Note that the function $g=g(\tau)$ is strictly monotonic from the monotonically increasing property of $\mu=\mu(|\tau|)$. After combining \eqref{Cn01} and \eqref{C2} it follows
\begin{align*}
\frac{1}{\int_{B_{R+t}(0)}\eta_q(t,t,x)\mathrm{d}x}	\langle t\rangle^{-1}\int_0^t (t-s) \langle s\rangle^{-1} \frac{[\mathcal{U}(s)]^{p_S(n)}}{\left(\log\langle s\rangle\right)^{p_S(n)-1}}    \mathrm{d}s\lesssim  g^{-1}\left(\frac{\ml{U}(t)}{\int_{B_{R+t}(0)}\eta_q(t,t,x)\mathrm{d}x} \right).
\end{align*}
The action of the mapping $g$ on both sides of the last estimate yields
\begin{align*}
	\ml{U}(t)&\gtrsim \int_{B_{R+t}(0)}\eta_q(t,t,x)\mathrm{d}x \  g\left[ \frac{1}{\int_{B_{R+t}(0)}\eta_q(t,t,x)\mathrm{d}x}\langle t\rangle^{-1}\int_0^t (t-s) \langle s\rangle^{-1} \frac{[\mathcal{U}(s)]^{p_S(n)}}{\left(\log\langle s\rangle\right)^{p_S(n)-1}}    \mathrm{d}s\right]\\
	&\gtrsim \langle t\rangle^{-1}\int_0^t (t-s) \langle s\rangle^{-1} \frac{[\mathcal{U}(s)]^{p_S(n)}}{\left(\log\langle s\rangle\right)^{p_S(n)-1}}    \mathrm{d}s \\
	&\ \quad\times\left[\mu\left(\,\left| \frac{1}{\int_{B_{R+t}(0)}\eta_q(t,t,x)\mathrm{d}x}\langle t\rangle^{-1}\int_0^t (t-s) \langle s\rangle^{-1} \frac{[\mathcal{U}(s)]^{p_S(n)}}{\left(\log\langle s\rangle\right)^{p_S(n)-1}}    \mathrm{d}s\right|\,\right)\right]^{\frac{1}{p_S(n)}}.
\end{align*}
Employing the non-negativity of $\ml{U}(t)$ stated in \eqref{Non-negativity-U} as well as
\begin{align}\label{Cnew-01}
	\int_{B_{R+t}(0)}\eta_q(t,t,x)\mathrm{d}x&\lesssim \langle t\rangle^{-\frac{n-1}{2}}\int_0^{R+t}\zeta^{n-1}\langle t-\zeta\rangle^{\frac{n-3}{2}-q}\mathrm{d}\zeta\notag\\
	&\lesssim\langle t\rangle^{\frac{n-1}{2}}\int_0^{R+t}\langle t-\zeta\rangle^{-1+\frac{1}{p_S(n)}}\mathrm{d}\zeta\leqslant C_1^{-1} \langle t\rangle^{\frac{n-1}{2}+\frac{1}{p_S(n)}}
\end{align}
from  Lemma \ref{lemma eta and xi estimates}, in conclusion, we obtain the iteration frame
\begin{align}\label{Iteration-Frame}
\ml{U}(t)&\geqslant C_0 \langle t\rangle^{-1}\int_0^t (t-s) \langle s\rangle^{-1} \frac{[\mathcal{U}(s)]^{p_S(n)}}{\left(\log\langle s\rangle\right)^{p_S(n)-1}}    \mathrm{d}s\notag\\
&\ \quad\times  \left[\mu\left(C_1 \langle t\rangle^{-\frac{n-1}{2}-\frac{1}{p_S(n)}}\langle t\rangle^{-1}\int_0^t (t-s) \langle s\rangle^{-1} \frac{[\mathcal{U}(s)]^{p_S(n)}}{\left(\log\langle s\rangle\right)^{p_S(n)-1}}    \mathrm{d}s\right)\right]^{\frac{1}{p_S(n)}}
\end{align}
for any $t\geqslant0$, with positive constants $C_0$ and $C_1$.

\subsection{Derivation of a first lower bound estimate}\label{Sub-3.3}
$\ \ \ \ $By applying \eqref{Eq-C-01} and the property (ii) in  Lemma \ref{lemma eta and xi estimates}, we may arrive at
\begin{align*}
	\int_{\mb{R}^n}u(t,x)\eta_q(t,t,x)\mathrm{d}x\gtrsim \langle t\rangle^{-1}\int_0^t(t-s)\langle s\rangle^{-q}\int_{\mb{R}^n}|u(s,x)|^{p_S(n)}\mu\big(|u(s,x)|\big)\mathrm{d}x\mathrm{d}s.
\end{align*}
An application of H\"older's inequality gives
\begin{align*}
	&\left|\int_{\mb{R}^n}u(s,x)\big[\mu\big(|u(s,x)|\big)\big]^{\frac{1}{p_S(n)}}\Psi(s,x)\mathrm{d}x\right|^{p_S(n)}\\
	 &\qquad\leqslant\int_{\mb{R}^n}|u(s,x)|^{p_S(n)}\mu\big(|u(s,x)|\big)\mathrm{d}x\left(\int_{B_{R+s}(0)}|\Psi(s,x)|^{p'_S(n)}\mathrm{d}x\right)^{\frac{p_S(n)}{p'_S(n)}}\\
	&\qquad\lesssim (R+s)^{(n-1)(\frac{p_S(n)}{2}-1)}\int_{\mb{R}^n}|u(s,x)|^{p_S(n)}\mu\big(|u(s,x)|\big)\mathrm{d}x,
\end{align*}
where we employed the next inequality (e.g. the proof was shown in \cite{Yordanov-Zhang=2006,Lai-Takamura=2019} by using an integration by parts):
\begin{align*}
	\int_{B_{R+s}(0)}|\Psi(s,x)|^{p'_S(n)}\mathrm{d}x\lesssim (R+s)^{(n-1)\left(1-\frac{1}{2}p'_S(n)\right)}.
\end{align*}
That is to say
\begin{align}
	&\int_{\mb{R}^n}u(t,x)\eta_q(t,t,x)\mathrm{d}x\notag\\
	&\qquad\gtrsim\langle t\rangle^{-1}\int_0^t(t-s)\langle s\rangle^{-1}\left|\int_{\mb{R}^n}u(s,x)\big[\mu\big(|u(s,x)|\big)\big]^{\frac{1}{p_S(n)}}\Psi(s,x)\mathrm{d}x\right|^{p_S(n)}\mathrm{d}s,\label{C-01}
\end{align}
due to the fact that
\begin{align*}
	-q-(n-1)\left(\frac{p_S(n)}{2}-1\right)=-\frac{1}{p_S(n)}\left(\frac{n-1}{2}p^2_S(n)-\frac{n-1}{2}p_S(n)-1 \right)=-1.
\end{align*}
Let us apply Lemma \ref{Lem-EGR} again with $g(\tau)=\tau [\mu(|\tau|)]^{\frac{1}{p_S(n)}}$ and $\alpha=\Psi(s,x)$ to arrive at
\begin{align}\label{C-02}
	\int_{\mb{R}^n}u(s,x)\big[\mu\big(|u(s,x)|\big)\big]^{\frac{1}{p_S(n)}}\Psi(s,x)\mathrm{d}x\geqslant \int_{B_{R+s}(0)}\Psi(s,x)\mathrm{d}x\,g\left(\frac{\int_{B_{R+s}(0)}u(s,x)\Psi(s,x)\mathrm{d}x}{\int_{B_{R+s}(0)}\Psi(s,x)\mathrm{d}x}\right).
\end{align}
A further step of integration by parts to \eqref{Eq.Defn.Energy.Solution.Crit.Case} shows
\begin{align*}
	&\int_{\mb{R}^n}u_t(t,x)\phi(t,x)\mathrm{d}x-\int_{\mb{R}^n}u(t,x)\phi_t(t,x)\mathrm{d}x\\
	 &\qquad=-\int_{\mb{R}^n}u_0(x)\phi_t(0,x)\mathrm{d}x+\int_{\mb{R}^n}u_1(x)\phi(0,x)\mathrm{d}x-\int_0^t\int_{\mb{R}^n}u(s,x)\big(\phi_{ss}(s,x)-\Delta\phi(s,x)\big)\mathrm{d}x\mathrm{d}s\\
	&\qquad\,\quad+\int_0^t\int_{\mb{R}^n}|u(s,x)|^{p_S(n)}\mu\big(|u(s,x)|\big)\phi(s,x)\mathrm{d}x\mathrm{d}s.
\end{align*}
Again, since $u$ is supported in a forward cone, we may apply the definition of energy solutions even though the test function is not compactly supported.
Taking as test function $\phi=\phi(t,x)$ the function $\Psi=\Psi(t,x)$, it holds
\begin{align*}
	 \frac{\mathrm{d}}{\mathrm{d}t}\int_{\mb{R}^n}u(t,x)\Psi(t,x)\mathrm{d}x+2\int_{\mb{R}^n}u(t,x)\Psi(t,x)\mathrm{d}x\geqslant\int_{\mb{R}^n}\big(u_0(x)+u_1(x)\big)\Phi(x)\mathrm{d}x
\end{align*}
due to the non-negativity of nonlinearity and $\Psi_{tt}=\Delta\Psi$. By multiplying $\mathrm{e}^{2t}$ on both sides of the last inequality, we can find
\begin{align*}
	 \int_{\mb{R}^n}u(t,x)\Psi(t,x)\mathrm{d}x\geqslant\frac{1}{2}\left(1+\mathrm{e}^{-2t}\right)\int_{\mb{R}^n}u_0(x)\Phi(x)\mathrm{d}x+\frac{1}{2}\left(1-\mathrm{e}^{-2t}\right)\int_{\mb{R}^n}u_1(x)\Phi(x)\mathrm{d}x.
\end{align*}
From our assumption on initial data, one gets
\begin{align}\label{C-03}
	\int_{B_{R+s}(0)}u(s,x)\Psi(s,x)\mathrm{d}x\gtrsim 1,
\end{align}
where the unexpressed multiplicative constant may depend on $u_0$ as well as $u_1$. With the aid of Lemma \ref{Lemma-Optimal-Est} and \eqref{C-03}, we are able to estimate from \eqref{C-02} and \eqref{C-01} that
\begin{align*}
	&\int_{\mb{R}^n}u(t,x)\eta_q(t,t,x)\mathrm{d}x\\
	&\qquad\gtrsim\langle t\rangle^{-1}\int_0^t(t-s)\langle s\rangle^{-1}\left|\int_{B_{R+s}(0)}\Psi(s,x)\mathrm{d}x\right|^{p_S(n)}\left|g
	\left(\frac{\int_{B_{R+s}(0)}u(s,x)\Psi(s,x)\mathrm{d}x}{\int_{B_{R+s}(0)}\Psi(s,x)\mathrm{d}x}\right)
	\right|^{p_S(n)}\mathrm{d}s\\
	&\qquad\gtrsim\langle t\rangle^{-1}\int_0^t(t-s)\langle s\rangle^{-1}(R+s)^{\frac{(n-1)p_S(n)}{2}}\left|g\left(C_2(R+s)^{-\frac{n-1}{2}}\right)\right|^{p_S(n)}\mathrm{d}s\\
	&\qquad\gtrsim\langle t\rangle^{-1}\int_0^t(t-s)\langle s\rangle^{-1}\mu\left(C_2(R+s)^{-\frac{n-1}{2}}\right)\mathrm{d}s
\end{align*}
with a positive constant $C_2$. According to \eqref{C2}, we derive
\begin{align*}
\langle t\rangle^{-1}\int_0^t(t-s)\langle s\rangle^{-1}\mu\left(C_2(R+s)^{-\frac{n-1}{2}}\right)\mathrm{d}s\lesssim  \int_{B_{R+t}(0)}\eta_q(t,t,x)\mathrm{d}x \  g^{-1}\left(\frac{\ml{U}(t)}{\int_{B_{R+t}(0)}\eta_q(t,t,x)\mathrm{d}x} \right).
\end{align*}
Furthermore, recalling the increasing property of $\mu$ and shrinking the interval of integration $[0,t]$ to $[1,t]$ for $t\geqslant 1$, one obtains
\begin{align*}
	\ml{U}(t)&\gtrsim \int_{B_{R+t}(0)}\eta_q(t,t,x)\mathrm{d}x\,g\left(\frac{1}{\int_{B_{R+t}(0)}\eta_q(t,t,x)\mathrm{d}x}\langle t\rangle^{-1}\int_0^t(t-s)\langle s\rangle^{-1}\mu\left(C_2(R+s)^{-\frac{n-1}{2}}\right)\mathrm{d}s\right)\\
%	&\gtrsim \langle t\rangle^{-1}\int_0^t(t-s)\langle s\rangle^{-1}\mu\left(C_2(R+s)^{-\frac{n-1}{2}}\right)\mathrm{d}s\\
%	&\quad\ \times \left[\mu\left(\frac{1}{\int_{B_{R+t}}\eta_q(t,t,x)\mathrm{d}x}\langle t\rangle^{-1}\int_0^t(t-s)\langle s\rangle^{-1}\mu\left(C_2(R+s)^{-\frac{n-1}{2}}\right)\mathrm{d}s\right)\right]^{\frac{1}{p_S(n)}}\\
	&\gtrsim \langle t\rangle^{-1}\mu\left(C_2(R+t)^{-\frac{n-1}{2}}\right)\int_1^t(t-s)\langle s\rangle^{-1}\mathrm{d}s\\
	&\quad\ \times \left[\mu\left(\frac{1}{\int_{B_{R+t}(0)}\eta_q(t,t,x)\mathrm{d}x}\langle t\rangle^{-1}\mu\left(C_2(R+t)^{-\frac{n-1}{2}}\right)\int_1^t(t-s)\langle s\rangle^{-1}\mathrm{d}s\right)\right]^{\frac{1}{p_S(n)}}.
\end{align*}
Taking account of
\begin{align*}
	\langle t\rangle^{-1}\int_1^t(t-s)\langle s\rangle^{-1}\mathrm{d}s&\gtrsim \langle t\rangle^{-1}\int_1^t\frac{t-s}{s}\mathrm{d}s=\langle t\rangle^{-1}\int_1^t\log s\,\mathrm{d}s\\
	&\gtrsim\frac{1}{3t}\int_{2t/3}^t\log s\,\mathrm{d}s\gtrsim\log\left(\frac{2t}{3}\right)
\end{align*}
%\color{red} The proof is correct, please read
%\begin{align*}
%\langle t\rangle^{-1}\int_1^t(t-s)\langle s\rangle^{-1}\mathrm{d}s=\frac{1}{3+t}\int_1^t\frac{t-s}{3+s}\mathrm{d}s\geqslant\frac{1}{4t}\int_1^t\frac{t-s}{4s}\mathrm{d}s=\frac{1}{16t}\int_1^t(t-s)\mathrm{d}\log s
%\end{align*}
%when $s,t\geqslant 1$. By using integration by parts, we have
%\begin{align*}
%\frac{1}{16t}\int_1^t(t-s)\mathrm{d}\log s=\frac{1}{16t}(t-s)\log s\big|_{s=1}^{s=t}+\frac{1}{16t}\int_1^t\log s\mathrm{d}s=\frac{1}{16t}\int_1^t\log s\mathrm{d}s\geqslant\frac{1}{16t}\int_{2t/3}^t\log s\mathrm{d}s
%\end{align*}
%when $t\geqslant\frac{3}{2}$. Thus,
%\begin{align*}
%\langle t\rangle^{-1}\int_1^t(t-s)\langle s\rangle^{-1}\mathrm{d}s\geqslant \frac{1}{16t}\int_{2t/3}^t\log s\mathrm{d}s\geqslant \frac{1}{16t}\log\left(\frac{2t}{3}\right)\int_{2t/3}^t\mathrm{d}s=\frac{1}{48}\log\left(\frac{2t}{3}\right).
%\end{align*}
for any $t\geqslant \frac{3}{2}$, and recalling \eqref{Cnew-01}, we may derive the lower bound estimate
\begin{align*}
		\ml{U}(t)\gtrsim \log\left(\frac{2t}{3}\right) \mu\left(C_2(R+t)^{-\frac{n-1}{2}}\right)\left[\mu\left(2C_3 \langle t\rangle^{-\frac{n-1}{2}-\frac{1}{p_S(n)}}\log\left(\frac{2t}{3}\right)\mu\left(C_2(R+t)^{-\frac{n-1}{2}}\right)\right)\right]^{\frac{1}{p_S(n)}}
\end{align*}
with a positive constant $C_3$.

Our assumption \eqref{Assumption-Blow-up} shows that there is a suitably large constant $c_l\gg1$ such that the next estimate holds:
\begin{align}\label{Special_Assum}
\mu(\tau)\left(\log\frac{1}{\tau}\right)^{\frac{1}{p_S(n)}}>\frac{c_l}{2} \ \ \mbox{for}\ \ 0<\tau\leqslant \tau_0\ll 1.
\end{align}
Let us choose a  large constant $t_0\geqslant 1$ such that for any $t\geqslant\frac{3}{2}t_0$, the following inequalities hold (later, we will take $t_0$ to be suitably large):
\begin{align}
C_2&\leqslant \tau_0(R+t)^{\frac{n-1}{2}} ,\label{Cnew-02}\\
%c_lC_3\log\left(\frac{2t}{3}\right)&\leqslant \tau_0\langle t\rangle^{\frac{n-1}{2}+\frac{1}{p_S(n)}}\left[\log\left(C_2^{-1}(R+t)^{\frac{n-1}{2}}\right)\right]^{\color{blue}\frac{1}{p_S(n)}},\notag\\
  c_l  C_2^{-1}C_3\left(\frac{n-1}{2}\right)^{-\frac{1}{p_S(n)}}\log\left(\frac{2t}{3}\right)&\geqslant\langle t\rangle^{\frac{n-1}{2}+\frac{1}{p_S(n)}}(R+t)^{-\frac{n-1}{2}-\frac{1}{p_S(n)}}
 \left[\log\left(C_2^{-\frac{2}{n-1}}(R+t)\right)\right]^{\frac{1}{p_S(n)}}.\notag
\end{align}
 Note that the second inequality  in the above (it will be used for reducing the argument) can be guaranteed since
\begin{align*}
\langle t\rangle^{\frac{n-1}{2}+\frac{1}{p_S(n)}}(R+t)^{-\frac{n-1}{2}-\frac{1}{p_S(n)}}\left[\log\left(C_2^{-\frac{2}{n-1}}(R+t)\right)\right]^{\frac{1}{p_S(n)}}\simeq [\log(R+t)]^{\frac{1}{p_S(n)}}\lesssim \log\left(\frac{2t}{3}\right)
\end{align*}
for large time $t$ since $p_S(n)>1$. According to our assumption \eqref{Assumption-Blow-up} for $t\geqslant \frac{3}{2}t_0$, it follows
\begin{align*}
&\mu\left(2C_3 \langle t\rangle^{-\frac{n-1}{2}-\frac{1}{p_S(n)}}\log\left(\frac{2t}{3}\right)\mu\left(C_2(R+t)^{-\frac{n-1}{2}}\right)\right)\\
&\qquad\geqslant\mu\left(c_lC_3\langle t\rangle^{-\frac{n-1}{2}-\frac{1}{p_S(n)}}\left[\log\left(C_2^{-1}(R+t)^{\frac{n-1}{2}}\right)\right]^{-\frac{1}{p_S(n)}}\log\left(\frac{2t}{3}\right)\right)\\
&\qquad\geqslant\mu\left(C_2(R+t)^{-\frac{n-1}{2}-\frac{1}{p_S(n)}}\right).
\end{align*}
Summarizing, we deduce the following first lower bound estimate:
\begin{align}\label{First-Lower-Bound}
	\ml{U}(t)\geqslant M_0 \log\left(\frac{2t}{3t_0}\right) \left[\mu\left(C_2(R+t)^{-\frac{n-1}{2}-\frac{1}{p_S(n)}-\epsilon_0}\right)\right]^{1+\frac{1}{p_S(n)}}
\end{align}
for any $t\geqslant \frac{3}{2}t_0$ with a large parameter $t_0\geqslant 1$, with a positive constant $M_0$ independent of $c_l$. Here, $\epsilon_0>0$ is a sufficiently small constant. We have to underline that such a small $\epsilon_0$ does not bring any influence on the blow-up condition.

\subsection{Iteration procedure and blow-up phenomenon: Proof of Theorem \ref{Thm-Blow-up}}\label{Sub-3.4}
$\ \ \ \ $Up to now, we have determined among other things the iteration frame \eqref{Iteration-Frame} for the functional $\ml{U}(t)$ and the first lower bound estimate \eqref{First-Lower-Bound} containing a logarithmic factor and a factor depending on the given modulus of continuity. In this part, we are going to prove a sequence of lower bound estimates for $\ml{U}(t)$ by applying the so-called slicing procedure, which has been introduced in \cite{Agemi-Kurokawa-Takamura=2000}.

Let us choose the sequence $\{\ell_j\}_{j\in\mb{N}_0}$ with $\ell_j:=2-2^{-(j+1)}$. Our goal is to derive the sequence of lower bound estimates for the functional $\ml{U}(t)$ as follows:
\begin{align}\label{Seq-Lower-Bound}
\ml{U}(t)\geqslant M_j(\log\langle t\rangle)^{-b_j}\left[\log\left(\frac{t}{\ell_{2j}t_0}\right)\right]^{a_j}\left[\mu\left(C_2(R+t)^{-\frac{n-1}{2}-\frac{1}{p_S(n)}-\epsilon_0}\right)\right]^{\sigma_j}
\end{align}
for $t\geqslant \ell_{2j}t_0$ with a suitably large constant $t_0\gg1$, where $\{M_j\}_{j\in\mb{N}_0}$, $\{a_j\}_{j\in\mb{N}_0}$, $\{b_j\}_{j\in\mb{N}_0}$ and $\{\sigma_j\}_{j\in\mb{N}_0}$ are sequences of non-negative real numbers that we shall determine recursively throughout the iteration procedure. From the first lower bound estimate \eqref{First-Lower-Bound}, we may choose with $j=0$ the parameters
\begin{align}\label{Initial_Exponent}
a_0:=1,\ \ b_0:=0,\ \ \sigma_0:=1+\frac{1}{p_S(n)}.
\end{align}
We are going to prove the validity of \eqref{Seq-Lower-Bound} for any $j\in\mb{N}_0$ by using an inductive argument. As we have already shown the validity of the basic case \eqref{First-Lower-Bound}, it remains to prove the inductive step. Let us assume that \eqref{Seq-Lower-Bound} holds for $j\geqslant 1$, our purpose is to demonstrate it for $j+1$.

First of all, via the lower bound estimate \eqref{Seq-Lower-Bound}, we know
\begin{align*}
&\langle t\rangle^{-1}\int_0^t (t-s) \langle s\rangle^{-1} \frac{[\mathcal{U}(s)]^{p_S(n)}}{\left(\log\langle s\rangle\right)^{p_S(n)-1}}    \mathrm{d}s\\
&\qquad\geqslant M_j^{p_S(n)}\langle t\rangle^{-1}\int_0^{t}\frac{t-s}{\langle s\rangle}\frac{\left[\log\left(\frac{s}{\ell_{2j}t_0}\right)\right]^{a_jp_S(n)}}{(\log\langle s\rangle )^{p_S(n)-1+b_jp_S(n)}}\left[\mu\left(C_2(R+s)^{-\frac{n-1}{2}-\frac{1}{p_S(n)}-\epsilon_0}\right)\right]^{\sigma_jp_S(n)}\mathrm{d}s\\
&\qquad\geqslant \frac{1}{3}M_j^{p_S(n)}\langle t\rangle^{-1}(\log\langle t\rangle)^{-(p_S(n)-1)-b_jp_S(n)}\left[\mu\left(C_2(R+t)^{-\frac{n-1}{2}-\frac{1}{p_S(n)}-\epsilon_0}\right)\right]^{\sigma_jp_S(n)}\\
&\qquad\quad\, \times\int_{\ell_{2j}t_0}^{t}\frac{t-s}{s}\left[\log\left(\frac{s}{\ell_{2j}t_0}\right)\right]^{a_jp_S(n)}\mathrm{d}s
\end{align*}
for $t\geqslant \ell_{2j+2}t_0$, where we shrank the interval of integration $[0,t]$ to $[\ell_{2j}t_0,t]$ so that $\langle s\rangle=3+|s|\leqslant 3s$ for any $s\geqslant \ell_{2j}t_0$. By employing integration by parts, we may derive
\begin{align*}
\int_{\ell_{2j}t_0}^{t}\frac{t-s}{s}\left[\log\left(\frac{s}{\ell_{2j}t_0}\right)\right]^{a_jp_S(n)}\mathrm{d}s&\geqslant\frac{1}{a_jp_S(n)+1}\int^t_{\frac{\ell_{2j}}{\ell_{2j+2}}t}\left[\log\left(\frac{s}{\ell_{2j}t_0}\right)\right]^{a_jp_S(n)+1}\mathrm{d}s\\
&\geqslant\frac{1}{3(a_jp_S(n)+1)}\left(1-\frac{\ell_{2j}}{\ell_{2j+2}}\right)\langle t\rangle\left[\log\left(\frac{t}{\ell_{2j+2}t_0}\right)\right]^{a_jp_S(n)+1}
\end{align*}
for $t\geqslant \ell_{2j+2}t_0$ so that $\ell_{2j}t_0\leqslant\frac{\ell_{2j}}{\ell_{2j+2}}t$. For this reason, the last relation for  $t\geqslant \ell_{2j+2}t_0$  implies immediately
\begin{align} \label{auxiliary estimate}
	&\langle t\rangle^{-1}\int_0^t (t-s) \langle s\rangle^{-1} \frac{[\mathcal{U}(s)]^{p_S(n)}}{\left(\log\langle s\rangle\right)^{p_S(n)-1}}    \mathrm{d}s \nonumber \\
	&\qquad\geqslant\frac{M_j^{p_S(n)}2^{-(2j+3)}}{3\ell_{2j+2}(a_jp_S(n)+1)}(\log\langle t\rangle)^{-(p_S(n)-1)-b_jp_S(n)}\left[\log\left(\frac{t}{\ell_{2j+2}t_0}\right)\right]^{a_jp_S(n)+1} \nonumber\\
	&\qquad\quad\,\times\left[\mu\left(C_2(R+t)^{-\frac{n-1}{2}-\frac{1}{p_S(n)}-\epsilon_0}\right)\right]^{\sigma_jp_S(n)}.
\end{align}

Plugging \eqref{Seq-Lower-Bound} into the iteration frame \eqref{Iteration-Frame}, we arrive after taking into consideration \eqref{auxiliary estimate} at
\begin{align}\label{Cnew-03}
\ml{U}(t)&\geqslant\frac{C_02^{-(2j+3)}M_j^{p_S(n)}}{3\ell_{2j+2}(a_jp_S(n)+1)}(\log\langle t\rangle)^{-(p_S(n)-1)-b_jp_S(n)}\left[\log\left(\frac{t}{\ell_{2j+2}t_0}\right)\right]^{a_jp_S(n)+1}\notag\\
&\quad\,\times\left[\mu\left(C_2(R+t)^{-\frac{n-1}{2}-\frac{1}{p_S(n)}-\epsilon_0}\right)\right]^{\sigma_jp_S(n)}[\ml{I}_{\mu}(t)]^{\frac{1}{p_S(n)}},
\end{align}
where we introduce
\begin{align*}
	\ml{I}_{\mu}(t)&:=\mu\left(C_1 \langle t\rangle^{-\frac{n-1}{2}-\frac{1}{p_S(n)}}\langle t\rangle^{-1}\int_0^t (t-s) \langle s\rangle^{-1} \frac{[\mathcal{U}(s)]^{p_S(n)}}{\left(\log\langle s\rangle\right)^{p_S(n)-1}}    \mathrm{d}s\right)
\end{align*} and estimate
\begin{align*}\ml{I}_{\mu}(t) &\geqslant \mu\bigg(\frac{C_12^{-(2j+3)}M_j^{p_S(n)}}{3\ell_{2j+2}(a_jp_S(n)+1)}\langle t\rangle^{-\frac{n-1}{2}-\frac{1}{p_S(n)}}(\log\langle t\rangle)^{-(p_S(n)-1)-b_jp_S(n)}\\
&\qquad\ \ \ \times\left[\mu\left(C_2(R+t)^{-\frac{n-1}{2}-\frac{1}{p_S(n)}-\epsilon_0}\right)\right]^{\sigma_jp_S(n)}\left[\log\left(\frac{t}{\ell_{2j+2}t_0}\right)\right]^{a_jp_S(n)+1}\bigg).
\end{align*}
Taking a suitably large $t_0$ such that for $t\geqslant \ell_{2j+2}t_0$, recalling the conclusion \eqref{Special_Assum} from  our assumption \eqref{Assumption-Blow-up} and the condition \eqref{Cnew-02}, then the following inequality holds:
%\begin{align*}
%	\tau_0&\geqslant\frac{c_l^{\sigma_jp_S(n)}C_12^{-(2j+3)}M_j^{p_S(n)}}{3\ell_{2j+2}(a_jp_S(n)+1)}\langle t\rangle^{-\frac{n-1}{2}-\frac{1}{p_S(n)}}(\log\langle t\rangle)^{-(p_S(n)-1)-b_jp_S(n)}\\
%	&\quad\ \times \left[\log\left(C_2^{-1}(R+t)^{\frac{n-1}{2}+\frac{1}{p_S(n)}+\epsilon_0}\right)\right]^{-\sigma_j}\left[\log\left(\frac{t}{\ell_{2j+2}t_0}\right)\right]^{a_jp_S(n)+1}
%\end{align*}
%as well as
\begin{align}\label{Verify-Chen}
	 &\frac{(\frac{c_l}{2})^{\sigma_jp_S(n)}C_1C_2^{-1}2^{-(2j+3)}M_j^{p_S(n)}}{3\ell_{2j+2}(a_jp_S(n)+1)}\left[\log\left(\frac{t}{\ell_{2j+2}t_0}\right)\right]^{a_jp_S(n)+1}\notag\\
	 &\geqslant \langle t\rangle^{\frac{n-1}{2}+\frac{1}{p_S(n)}}(R+t)^{-\frac{n-1}{2}-\frac{1}{p_S(n)}-\epsilon_0}(\log\langle t\rangle)^{p_S(n)-1+b_jp_S(n)}\left[\log\left(C_2^{-1}(R+t)^{\frac{n-1}{2}+\frac{1}{p_S(n)}+\epsilon_0}\right)\right]^{\sigma_j}
\end{align}
for a fixed $j$, because the polynomial decay factor  $(R+t)^{-\epsilon_0}$ plays from the point of decay a dominant role in comparison with all logarithmic factors on the right-hand side of the last inequality. Later, we will verify the last inequality \eqref{Verify-Chen} uniformly for all $j\gg1$ by choosing suitable parameters $a_j,b_j,\sigma_j$ and estimating $M_j$. According to the last lower bounds estimates, it provides
\begin{align}\label{Cnew-04}
	\ml{I}_{\mu}(t)\geqslant \mu\left(C_2(R+t)^{-\frac{n-1}{2}-\frac{1}{p_S(n)}-\epsilon_0}\right)
\end{align}
for $t\geqslant \ell_{2j+2}t_0$ with a suitably large $t_0\gg1$. Summarizing the above estimates \eqref{Cnew-03} as well as \eqref{Cnew-04}, we claim the lower bound estimate
\begin{align*}
	\ml{U}(t)&\geqslant\frac{C_02^{-(2j+3)}M_j^{p_S(n)}}{3\ell_{2j+2}(a_jp_S(n)+1)}(\log\langle t\rangle)^{-(p_S(n)-1)-b_jp_S(n)}\left[\log\left(\frac{t}{\ell_{2j+2}t_0}\right)\right]^{a_jp_S(n)+1}\\
	&\quad\,\times\left[\mu\left(C_2(R+t)^{-\frac{n-1}{2}-\frac{1}{p_S(n)}-\epsilon_0}\right)\right]^{\sigma_jp_S(n)+\frac{1}{p_S(n)}}
\end{align*}
for $t\geqslant \ell_{2j+2}t_0$. In other words, we have proved \eqref{Seq-Lower-Bound} for $j+1$ provided that
\begin{align*}
&\qquad \qquad \qquad \quad \  M_{j+1}:=\frac{C_02^{-(2j+3)}}{3\ell_{2j+2}(a_jp_S(n)+1)}M_j^{p_S(n)},\\
 &a_{j+1}:=1+a_jp_S(n),\ \ b_{j+1}:=p_S(n)-1+b_jp_S(n),\ \ \sigma_{j+1}:=\frac{1}{p_S(n)}+\sigma_jp_S(n).
\end{align*}
By using recursively the relations and the initial exponents \eqref{Initial_Exponent}, we deduce
\begin{align}\label{absig}
	&a_j=\frac{p_S(n)}{p_S(n)-1}p_S^j(n)-\frac{1}{p_S(n)-1},\ \ b_j=p_S^j(n)-1,\notag\\
	&\ \ \ \ \quad  \sigma_j=\frac{p_S(n)}{p_S(n)-1}p_S^j(n)-\frac{1}{(p_S(n)-1)p_S(n)}.
\end{align}
Due to the facts that $\ell_{2j}\leqslant 2$ and $a_j\leqslant\frac{p_S(n)}{p_S(n)-1}p_S^j(n)$, the lower bound of $M_j$ can be estimated by
\begin{align*}
	M_j=\frac{C_0 2^{-(2j+1)}}{3\ell_{2j}(1+a_{j-1}p_S(n))}M^{p_S(n)}_{j-1}\geqslant C_4 [4p_S(n)]^{-j}M_{j-1}^{p_S(n)}
\end{align*}
with the constant $C_4:=\frac{C_0(p_S(n)-1)}{12 p_S(n)}>0$, which depends on $n$ but is independent of $j$. Applying the logarithmic function to both sides of the last inequality and using iteratively the resulting inequality, we may obtain
\begin{align*}
\log M_j&\geqslant p_S(n)\log M_{j-1}-j\log[4p_S(n)]+\log C_4\\
&\geqslant\cdots\geqslant p_S^j(n)\log M_0-\left(\sum\limits_{k=0}^{j-1}(j-k)p_S^k(n)\right)\log[4p_S(n)]+\left(\sum\limits_{k=0}^{j-1}p_S^k(n)\right)\log C_4\\
&\qquad \ \  \geqslant p_S^j(n)\left(\log M_0-\frac{p_S(n)\log[4p_S(n)]}{[p_S(n)-1]^2}+\frac{\log C_4}{p_S(n)-1}\right)\\
&\qquad\quad\ \ \  +\frac{j[p_S(n)-1]+p_S(n)}{[p_S(n)-1]^2}\log[4p_S(n)]-\frac{\log C_4}{p_S(n)-1},
\end{align*}
where we used the identities
\begin{align*}
\sum\limits_{k=0}^{j-1}(j-k)p^k=\frac{1}{p-1}\left(\frac{p^{j+1}-p}{p-1}-j\right)\ \ \mbox{and}\ \ \sum\limits_{k=0}^{j-1}p^k=\frac{p^j-1}{p-1}.
\end{align*}
Let us define $j_1=j_1(p_S(n))$ as the smallest non-negative integer such that
\begin{align*}
j_1\geqslant\frac{\log C_4}{\log[4p_S(n)]}-\frac{p_S(n)}{p_S(n)-1}.
\end{align*}
We may estimate
\begin{align}\label{LOGM}
	\log M_j\geqslant p_S^j(n)\left(\log M_0-\frac{p_S(n)\log[4p_S(n)]}{[p_S(n)-1]^2}+\frac{\log C_4}{p_S(n)-1}\right)=p_S^j(n)\log C_5
\end{align}
carrying the positive constant
\begin{align*}
	C_5:= M_0[4p_S(n)]^{-\frac{p_S(n)}{[p_S(n)-1]^2}} C_4^{\frac{1}{p_S(n)-1}},
\end{align*}
which depends on $n$ but is independent of $j$.

Let us now prove the inequality \eqref{Verify-Chen} uniformly for all $j\gg1$.
Due to the choices of parameters $a_j,b_j,\sigma_j$, we may estimate
\begin{align*}
&\langle t\rangle^{\frac{n-1}{2}+\frac{1}{p_S(n)}}(R+t)^{-\frac{n-1}{2}-\frac{1}{p_S(n)}-\epsilon_0}(\log\langle t\rangle)^{p_S(n)-1+b_jp_S(n)}\\
&\times\left[\log\left(C_2^{-1}(R+t)^{\frac{n-1}{2}+\frac{1}{p_S(n)}+\epsilon_0}\right)\right]^{\sigma_j} \left[\log\left(\frac{t}{\ell_{2j+2}t_0}\right)\right]^{-a_jp_S(n)-1}\\
&\qquad\leqslant C(R+t)^{-\epsilon_0}(\log\langle t\rangle)^{p_S(n)-2+b_jp_S(n)+\sigma_j-a_jp_S(n)}\\
&\qquad\leqslant C(R+t)^{-\epsilon_0}(\log\langle t\rangle)^{-1+\frac{1}{p_S(n)}}.
\end{align*}
Moreover, from the lower bound of $M_j$ in \eqref{LOGM}, we know
\begin{align*}
\frac{(\frac{c_l}{2})^{\sigma_jp_S(n)}C_1C_2^{-1}2^{-(2j+3)}M_j^{p_S(n)}}{3\ell_{2j+2}(a_jp_S(n)+1)}\geqslant\frac{\widetilde{C}}{(p_S^{j+1}(n)-1)4^j}
\left[(\frac{c_l}{2})^{\frac{p_S(n)}{p_S(n)-1}}C_5\right]^{p_S^{j+1}(n)},
\end{align*}
where $\widetilde{C}$ is a positive constant independent of $j$. Note that $C_5$ only depends on $M_0,C_0,p_S(n)$ but it is independent of $c_l$. Taking a suitably large constant $c_l$, in the last two inequalities, we notice
\begin{align*}
\frac{\widetilde{C}}{(p_S^{j+1}(n)-1)4^j}\left[\Big(\frac{c_l}{2}\Big)^{\frac{p_S(n)}{p_S(n)-1}}C_5\right]^{p_S^{j+1}(n)}\geqslant  C(R+t)^{-\epsilon_0}(\log\langle t\rangle)^{-1+\frac{1}{p_S(n)}}
\end{align*}
for any $j\gg1$. Thus, we verified \eqref{Verify-Chen} uniformly for all $j\gg1$.

Consequently, recalling the sequence of estimates \eqref{Seq-Lower-Bound} associated with \eqref{absig}, \eqref{LOGM} and $\ell_{2j}\leqslant 2$, the lower bound estimate of $\ml{U}(t)$ can be presented as follows:
\begin{align*}
\ml{U}(t)&\geqslant \exp\left(p_S^j(n)\log C_5\right)(\log\langle t\rangle)^{-p_S^j(n)+1}\left[\log\left(\frac{t}{2t_0}\right)\right]^{\frac{p_S(n)}{p_S(n)-1}p_S^j(n)-\frac{1}{p_S(n)-1}}\\
&\quad\ \times\left[\mu\left(C_2(R+t)^{-\frac{n-1}{2}-\frac{1}{p_S(n)}-\epsilon_0}\right)\right]^{\frac{p_S(n)}{p_S(n)-1}p_S^j(n)-\frac{1}{(p_S(n)-1)p_S(n)}}\\
&\geqslant\exp\left\{p_S^j(n)\log\left[C_5(\log\langle t\rangle)^{-1}\left[\log\left(\frac{t}{2t_0}\right)\right]^{\frac{p_S(n)}{p_S(n)-1}}\left[\mu\left(C_2(R+t)^{-\frac{n-1}{2}-\frac{1}{p_S(n)}-\epsilon_0}\right)\right]^{\frac{p_S(n)}{p_S(n)-1}} \right] \right\}\\
&\quad\ \times\log\langle t\rangle\left[\log\left(\frac{t}{2t_0}\right)\right]^{-\frac{1}{p_S(n)-1}}\left[\mu\left(C_2(R+t)^{-\frac{n-1}{2}-\frac{1}{p_S(n)}-\epsilon_0}\right)\right]^{-\frac{1}{(p_S(n)-1)p_S(n)}}
\end{align*}
for $t\geqslant 2t_0$ and any $j\geqslant j_1$. There exists a constant $C_6>0$ such that for  $t\geqslant t_1$ with a suitably large constant $t_1$, it holds
\begin{align*}
	C_2(R+t)^{-\frac{n-1}{2}-\frac{1}{p_S(n)}-\epsilon_0}\geqslant (C_6t)^{-\frac{n-1}{2}-\frac{1}{p_S(n)}-\epsilon_0}.
\end{align*}
Moreover, concerning suitably large $t\geqslant t_2$, the following estimates hold:
\begin{align*}
	\log\langle t\rangle=\log(3+t)\leqslant 2\log(C_6t),\ \
	\log\left(\frac{t}{2t_0}\right)\geqslant\frac{1}{2t_0}\log(C_6t),\ \ (C_6t)^{-\frac{n-1}{2}-\frac{1}{p_S(n)}-\epsilon_0}\leqslant \tau_0.
\end{align*}
For $t\geqslant\max\{2t_0,t_1,t_2\}$, the lower bound of the functional can be controlled by
\begin{align}\label{Final-Lower-Bound}
	 \ml{U}(t)&\geqslant\exp\left\{p_S^j(n)\log\left[C_52^{-1}(2t_0)^{-\frac{p_S(n)}{p_S(n)-1}}[\log(C_6t)]^{\frac{1}{p_S(n)-1}}\left[\mu\left((C_6t)^{-\frac{n-1}{2}-\frac{1}{p_S(n)}-\epsilon_0}\right)\right]^{\frac{p_S(n)}{p_S(n)-1}} \right] \right\}\notag\\
	&\quad\ \times\log\langle t\rangle\left[\log\left(\frac{t}{2t_0}\right)\right]^{-\frac{1}{p_S(n)-1}}\left[\mu\left(C_2(R+t)^{-\frac{n-1}{2}-\frac{1}{p_S(n)}-\epsilon_0}\right)\right]^{-\frac{1}{(p_S(n)-1)p_S(n)}}.
\end{align}
Recalling our assumption \eqref{Assumption-Blow-up} it follows that there exists a positive, continuous function $\kappa=\kappa(\tau)$ such that
\begin{align}\label{Property-kappa}
	\lim\limits_{\tau\to0^+}\kappa(\tau)\in [c_l,+\infty]\ \ \mbox{and}\ \ \mu(\tau) =\left(\log\frac{1}{\tau}\right)^{-\frac{1}{p_S(n)}}\kappa(\tau).
\end{align}
When $t\geqslant\max\{2t_0,t_1,t_2\}$, the crucial part in \eqref{Final-Lower-Bound} is expressed by
\begin{align*}
	 \log(C_6t)\left[\mu\left((C_6t)^{-\frac{n-1}{2}-\frac{1}{p_S(n)}-\epsilon_0}\right)\right]^{p_S(n)}=\left(\frac{n-1}{2}+\frac{1}{p_S(n)}+\epsilon_0\right)^{-1}\left[\kappa\left((C_6t)^{-\frac{n-1}{2}-\frac{1}{p_S(n)}-\epsilon_0}\right)\right]^{p_S(n)}
\end{align*}
and the lower bound estimate turns into
\begin{align*}
		 \ml{U}(t)&\geqslant\exp\left\{p_S^j(n)\log\left[C_7\,\left[\kappa\left((C_6t)^{-\frac{n-1}{2}-\frac{1}{p_S(n)}-\epsilon_0}\right)\right]^{\frac{p_S(n)}{p_S(n)-1}} \right] \right\}\log\langle t\rangle\left[\log\left(\frac{t}{2t_0}\right)\right]^{-\frac{1}{p_S(n)-1}}\\
		&\quad\ \times\left[\mu\left(C_2(R+t)^{-\frac{n-1}{2}-\frac{1}{p_S(n)}-\epsilon_0}\right)\right]^{-\frac{1}{(p_S(n)-1)p_S(n)}}
\end{align*}
with a suitable constant $C_7>0$ independent of $j$. Let us take account of the property \eqref{Property-kappa} of $\kappa=\kappa(\tau)$, namely,
\begin{align*}
\kappa(\tau)>\frac{c_l}{2}\ \ \mbox{when}\ \ 0<\tau\ll 1	
\end{align*}
 with a suitably large constant $c_l>2(100/C_7)^{[p_S(n)-1]/p_S(n)}$, where $C_7$ only depends on $M_0$, $C_0$, $p_S(n)$, $n$ and is independent of $j$. Then for large time $t\geqslant\max\{2t_0,t_1,t_2\}$, we find that in the last estimate
\begin{align*}
\log\left[C_7[\kappa(\tau)]^{\frac{p_S(n)}{p_S(n)-1}}\right]>\log\left[C_7(c_l/2)^{\frac{p_S(n)}{p_S(n)-1}}\right]>1
\end{align*}
for $0<\tau=(C_6t)^{-\frac{n-1}{2}-\frac{1}{p_S(n)}-\epsilon_0}\ll 1$. Finally, taking the limit as $j\to+\infty$ in the above estimate the lower bound for $\ml{U}(t)$ blows up in finite time. This completes the proof of Theorem \ref{Thm-Blow-up}. \hfill $\Box$

\section{Global (in time) existence of radial solutions in three dimensions}\label{Section-GESDS}
\setcounter{equation}{0}
$\ \ \ \ $Firstly, by introducing a polynomial-logarithmic type weighted Banach space, we will prepare uniform bounded $L^{\infty}$ estimates of the radial solution to the three dimensional free wave equation in Subsection \ref{Sub-4.1}. Then, the philosophy of the proof for Theorem \ref{Thm-GESDS} and its key tool will be stated in Subsection \ref{Sub-4.2}. We will demonstrate the global (in time) existence result in Subsection \ref{Sub-4.3} by applying a refined  analysis  in the $(t,r)$-plane to estimate the nonlinear terms.
\subsection{Preliminary and weighted $L^{\infty}_tL^{\infty}_r$ estimates for the linearized model}\label{Sub-4.1}
$\ \ \ \ $As preparations for studying nonlinear models, we will state some polynomial-logarithmic type weighted $L^{\infty}_tL^{\infty}_r$ estimates  for the linear wave equation in the radial case. Let us first extend initial data $u_0(r)$ and $u_1(r)$ by even reflections, namely,
\begin{align*}
	u_0(-r)=u_0(r)\ \ \mbox{and}\ \ u_1(-r)=u_1(r)\ \ \mbox{for}\ \ r<0.
\end{align*}
Note that our assumptions $u_0\in\ml{C}^2$ and $u_0$ radially symmetric ensure $u_0'(0)=0$. Due to our interest of radial solutions and the application of even reflections, we may rewrite the semilinear Cauchy problem \eqref{Semilinear-Wave-Modulus} in three dimensions as
\begin{align}\label{G1}
	\begin{cases}
		\displaystyle{u_{tt}-u_{rr}-\frac{2}{r}u_r=|u|^{p_S(3)}\mu(|u|),}&r\in\mb{R},\ t>0,\\
		u(0,r)=u_0(r),\ \ u_t(0,r)=u_1(r),&r\in\mb{R}.
	\end{cases}
\end{align}
%\begin{defn}\label{Defn-Global-solution}
%	We say that $u=u(t,r)$ is a global (in time) radial solution to the semilinear Cauchy problem \eqref{G1} if $u\in\ml{C}([0,+\infty)\times\mb{R})$, $ru\in\ml{C}^1([0,+\infty)\times\mb{R})$, $r^2 u\in\ml{C}^2([0,+\infty)\times\mb{R})$ and
%	\begin{align*}
%		\begin{cases}
%			r^2u_{tt}-r^2u_{rr}-2ru_r=r^2|u|^{p_S(3)}\mu(|u|),&r\in\mb{R},\ t>0,\\
%			u(0,r)=u_0(r),\ \ u_t(0,r)=u_1(r),&r\in\mb{R}.
%		\end{cases}
%	\end{align*}
%\end{defn}
%\begin{remark}
%Let us underline explicitly that any solution to the semilinear Cauchy problem \eqref{G1} in the sense of Definition \ref{Defn-Global-solution} gives a solution in the space $\ml{C}([0,+\infty)\times\mb{R}^3)$ to the original problem \eqref{Semilinear-Wave-Modulus} in three dimensions.
%\end{remark}

 Now we turn our focus to the linear model with vanishing right-hand side and the same initial data as those of \eqref{G1}, namely,
\begin{align}\label{Linearized-v}
	\begin{cases}
	\displaystyle{v_{tt}-v_{rr}-\frac{2}{r}v_r=0,}&r\in\mb{R},\ t>0,\\
	v(0,r)=u_0(r),\ \ v_t(0,r)=u_1(r),&r\in\mb{R}.
\end{cases}
\end{align}
Let us recall $u_0\in\ml{C}^2_0$ as well as $u_1\in\ml{C}^1_0$. According to the well-known d'Alembert's formula, the solution to the linear Cauchy problem \eqref{Linearized-v} can be represented as follows:
\begin{align}\label{Rep-linearized}
	v(t,r)&=\frac{\partial}{\partial t}\left(\,\int_{-1}^1H_{u_0}(t+r\sigma)\mathrm{d}\sigma\right)+\int_{-1}^1H_{u_1}(t+r\sigma)\mathrm{d}\sigma\notag\\
	&=\frac{1}{2r}\big((t+r)u_0(t+r)-(t-r)u_0(t-r)\big)+\frac{1}{r}\int_{t-r}^{t+r}H_{u_1}(\rho)\mathrm{d}\rho,
\end{align}
where we denoted
\begin{align*}
	H_{u_j}(\rho):=\frac{\rho}{2}u_j(\rho)\ \ \mbox{with}\ \ j=0,1.
\end{align*}
Its proof is standard by taking the new variable $rv(t,r)$ and the representation of solution for the one dimensional free wave equation.

Motivated by the papers \cite{Asakura=1986,Kubo=1997,Dabbicco-Lucente-Reissig=2015}, we are able to derive some decay estimates in some weighted $L^{\infty}_t L^{\infty}_r$ spaces for radial solutions of the linear Cauchy problem \eqref{Linearized-v}. In order to overcome some difficulties from the influence of modulus of continuity when we consider the nonlinear model \eqref{G1}, we will include an additional logarithmic type weighted function in the solution space. To be specific, we introduce the Banach space
\begin{align*}
	X_{\kappa}:=\big\{v\in\ml{C}([0,+\infty)\times\mb{R}):\ v \mbox{ is even in }r \ \ \mbox{and}\ \ \|v\|_{X_{\kappa}}<+\infty\big\}
\end{align*}
 with a polynomial-logarithmic type weighted norm
\begin{align*}
	\|v\|_{X_{\kappa}}:=\sup\limits_{t\geqslant 0,\ r\in\mb{R}}\big(\omega(\langle t-|r|\rangle)\langle t+|r|\rangle\langle t-|r|\rangle^{\kappa-1}|v(t,r)|\big)\ \ \mbox{with}  \ \ \kappa:=1+\frac{1}{p_S(3)}.
\end{align*}
Here,  the new weighted factor is defined by
\begin{align}\label{Weighted-function}
	\omega(\tau):=(\log\tau)^{\frac{1}{p_S(3)}} \ \ \mbox{for any}\ \ \tau \geqslant 3.
\end{align}
 Then, we have the next result for bounded estimates in the Banach space $X_{\kappa}$.
\begin{prop}\label{Prop-linear-decay}
	Let $u_0\in \ml{A}_{\kappa}\cap\ml{C}^2$ and $u_1\in \ml{B}_{\kappa+1}\cap \ml{C}^1$ with $\kappa>1$. Then, the following estimate holds:
	\begin{align*}
		\|v\|_{X_{\kappa}}\lesssim
		\|u_0\|_{\ml{A}_{\kappa}}+\|u_1\|_{\ml{B}_{\kappa+1}},
	\end{align*}
	where the Banach spaces for initial data are defined as follows:
	\begin{align*}
		\ml{A}_{\kappa}&:=\big\{h\in\ml{C}^1:\ h\mbox{ is an even function and }\|h\|_{\ml{A}_{\kappa}}<+\infty  \big\},\\
		\ml{B}_{\kappa}&:=\big\{h\in\ml{C}:\ h\mbox{ is an even function and }\|h\|_{\ml{B}_{\kappa}}<+\infty  \big\},
	\end{align*}
	carrying the corresponding norms
	\begin{align*}
		\|h\|_{\ml{A}_{\kappa}}&:=\sup\limits_{r\geqslant 0}\big(\omega(\langle r\rangle)\langle r\rangle^{\kappa}|h(r)|\big)+\sup\limits_{r\geqslant0}\big(\omega(\langle r\rangle)\langle r\rangle^{\kappa+1}|h'(r)|\big),\\
		\|h\|_{\ml{B}_{\kappa}}&:=\sup\limits_{r\geqslant 0}\big(\omega(\langle r\rangle)\langle r\rangle^{\kappa}|h(r)|\big).
	\end{align*} In the above, we used the property to be even for the functions $\langle r\rangle$ and $|h(r)|,|h'(r)|$ so that we just need to consider $r\geqslant0$ in these norms.
\end{prop}
\begin{proof}
By using the definitions of $\|u_0\|_{\ml{A}_{\kappa}}$ and $\|u_1\|_{\ml{B}_{\kappa+1}}$, respectively, we may estimate
\begin{align*}
	|H_{u_0}(\rho)|&\lesssim
	[\omega(\langle \rho\rangle)]^{-1}\langle\rho\rangle^{1-\kappa}\|u_0\|_{\ml{A}_{\kappa}},\\
	|H'_{u_0}(\rho)|&\lesssim |u_0(\rho)|+\langle\rho\rangle|u_0'(\rho)|\lesssim [\omega(\langle \rho\rangle)]^{-1}\langle \rho\rangle^{-\kappa}\|u_0\|_{\ml{A}_{\kappa}},\\
	|H_{u_1}(\rho)|&\lesssim
	[\omega(\langle \rho\rangle)]^{-1}\langle\rho\rangle^{-\kappa}\|u_1\|_{\ml{B}_{\kappa+1}}.
\end{align*}
Let us employ the triangle inequality to observe
\begin{align}\label{Tool_01}
	\langle t-|r|\rangle=3+\big|t-|r|\big|\leqslant3+|t\pm r|=\langle t\pm r\rangle.
\end{align}
 According to the solution formula \eqref{Rep-linearized}, one may derive
\begin{align*}
	|v(t,r)|&\leqslant\frac{1}{|r|}\big(|H_{u_0}(t+r)|+|H_{u_0}(t-r)|\big)+\frac{1}{|r|}\int_{t-|r|}^{t+|r|}|H_{u_1}(\rho)|\mathrm{d}\rho\\
	&\lesssim\frac{1}{|r|}\Big([\omega(\langle t+r\rangle)]^{-1}\langle t+r\rangle^{1-\kappa}+[\omega(\langle t-r\rangle)]^{-1}\langle t-r\rangle^{1-\kappa}\Big)\left(\|u_0\|_{\ml{A}_{\kappa}}+\|u_1\|_{\ml{B}_{\kappa+1}}\right)\\
	&\quad+\frac{1}{|r|}\int_{t-|r|}^{t+|r|}[\omega(\langle \rho\rangle)]^{-1}\langle \rho\rangle^{-\kappa}\mathrm{d}\rho\left(\|u_0\|_{\ml{A}_{\kappa}}+\|u_1\|_{\ml{B}_{\kappa+1}}\right)\\
	&\lesssim\left(\frac{1}{|r|}[\omega(\langle t-|r|\rangle)]^{-1}\langle t-|r|\rangle^{1-\kappa}+\frac{1}{|r|}\int_{t-|r|}^{t+|r|}[\omega(\langle \rho\rangle)]^{-1}\langle \rho\rangle^{-\kappa}\mathrm{d}\rho\right)\left(\|u_0\|_{\ml{A}_{\kappa}}+\|u_1\|_{\ml{B}_{\kappa+1}}\right),
\end{align*}
because of $\kappa>1$, where we used the relation \eqref{Tool_01}. Let us separate our next consideration into two situations with respect to the interplay between $t$ and $|r|$.
	\begin{itemize}
	\item When $t\geqslant 2|r|$, since the integrand takes its maximum for $\rho=t-|r|$ and $\langle t+|r|\rangle\approx\langle t-|r|\rangle$, thanks to the representation \eqref{Rep-linearized}, we may estimate
	\begin{align*}
		&\omega(\langle t-|r|\rangle)\langle t+|r|\rangle\langle t-|r|\rangle^{\kappa-1}|v(t,r)|\\
		&\qquad\lesssim\frac{\omega(\langle t-|r|\rangle)\langle t+|r|\rangle\langle t-|r|\rangle^{\kappa-1}}{|r|}|H_{u_0}(t+r)-H_{u_0}(t-r)|+\frac{\langle t+|r|\rangle}{\langle t-|r|\rangle}\|u_1\|_{\ml{B}_{\kappa+1}}\\
		&\qquad\lesssim \omega(\langle t-|r|\rangle)\langle t+|r|\rangle\langle t-|r|\rangle^{\kappa-1}|H'_{u_0}(\zeta)|+\|u_1\|_{\ml{B}_{\kappa+1}}\\
		&\qquad\lesssim \|u_0\|_{\ml{A}_{\kappa}}+\|u_1\|_{\ml{B}_{\kappa+1}},
	\end{align*}
	where we employed the mean value theorem with $\zeta\in(t-r,t+r)$, \eqref{Tool_01} and
	\begin{align*}
		|H'_{u_0}(\zeta)|&\lesssim [\omega(\langle\zeta\rangle)]^{-1}\langle\zeta\rangle^{-\kappa}\|u_0\|_{\ml{A}_{\kappa}}\lesssim [\omega(\langle t-|r|\rangle)]^{-1}\langle t-|r|\rangle^{-\kappa}\|u_0\|_{\ml{A}_{\kappa}}.
	\end{align*}
	\item When $t\leqslant 2|r|$, we have some further discussions.
	\begin{itemize}
		\item If $|r|\leqslant 1$, since $\langle t-|r|\rangle\approx\langle t+|r|\rangle\approx 3$ and the compact $(t,r)$-zone, then we get
		\begin{align*}
			\omega(\langle t-|r|\rangle)\langle t+|r|\rangle\langle t-|r|\rangle^{\kappa-1}|v(t,r)|\lesssim\|u_0\|_{\ml{A}_{\kappa}}+\|u_1\|_{\ml{B}_{\kappa+1}},
		\end{align*}
		whose approach is the same as the one for $t\geqslant 2|r|$.
		\item If $|r|\geqslant 1$, since $\langle t+|r|\rangle\leqslant 3\langle r\rangle$ and $|r|\approx \langle r\rangle$, then we get
		\begin{align*}
			&\omega(\langle t-|r|\rangle)\langle t+|r|\rangle\langle t-|r|\rangle^{\kappa-1}|v(t,r)|\\
			&\qquad\lesssim \left(\frac{\langle t+|r|\rangle}{|r|}+\frac{\langle t+|r|\rangle}{|r|}\langle t-|r|\rangle^{\kappa-1}\int_{t-|r|}^{t+|r|}\langle\rho\rangle^{-\kappa}\mathrm{d}\rho\right)\left(\|u_0\|_{\ml{A}_{\kappa}}+\|u_1\|_{\ml{B}_{\kappa+1}}\right) \\
			&\qquad\lesssim \left(1+\frac{\langle t-|r|\rangle^{\kappa-1}}{\kappa-1}\left(\langle t-|r|\rangle^{1-\kappa}-\langle t+|r|\rangle^{1-\kappa}\right)\right)\left(\|u_0\|_{\ml{A}_{\kappa}}+\|u_1\|_{\ml{B}_{\kappa+1}}\right)\\
			&\qquad\lesssim \|u_0\|_{\ml{A}_{\kappa}}+\|u_1\|_{\ml{B}_{\kappa+1}},
		\end{align*}
	because of $\kappa>1$.
	\end{itemize}
\end{itemize}
In other words, we arrive at
\begin{align*}
	\left\|\omega(\langle t-|r|\rangle)\langle t+|r|\rangle\langle t-|r|\rangle^{\kappa-1}v(t,r)\right\|_{L^{\infty}([0,+\infty)\times\mb{R})}\lesssim \|u_0\|_{\ml{A}_{\kappa}}+\|u_1\|_{\ml{B}_{\kappa+1}},
\end{align*}
which completes our proof.
\end{proof}

\subsection{Philosophy of our approach}\label{Sub-4.2}
$\ \ \ \ $By Duhamel's principle, the solution to the inhomogeneous linear Cauchy problem with $F=F(t,r)$ as a source term and vanishing data is given by
\begin{align*}
	 v^{\mathrm{inh}}(t,r)&=\int_0^t\int_{-1}^1H_F[s](t-s+r\sigma)\mathrm{d}\sigma\mathrm{d}s=\frac{1}{r}\int_0^t\int_{t-s-r}^{t-s+r}H_F[s](\rho)\mathrm{d}\rho\mathrm{d}s
\end{align*}
with $H_F[s](\rho):=\frac{\rho}{2}F(s,\rho)$.
Concerning the semilinear Cauchy problem \eqref{G1}, inspired by the last representation, we introduce
\begin{align*}
	 Lu(t,r):=\int_0^t\int_{-1}^1H_u[s](t-s+r\sigma)\mathrm{d}\sigma\mathrm{d}s=\frac{1}{r}\int_0^t\int_{t-s-r}^{t-s+r}H_u[s](\rho)\mathrm{d}\rho\mathrm{d}s
\end{align*}
with the nonlinear term
\begin{align*}
	H_u[s](\rho):=\frac{\rho}{2}|u(s,\rho)|^{p_S(3)}\mu\big(|u(s,\rho)|\big).
\end{align*}

In view of Duhamel's principle as well as the above setting, we expect that if we find $u\in X_{\kappa}$ with $\kappa>1$ such that
\begin{align*}
	u(t,r)&=v(t,r)+Lu(t,r)\\
	&=\frac{\partial}{\partial t}\left(\,\int_{-1}^1H_{u_0}(t+r\sigma)\mathrm{d}\sigma\right)+\int_{-1}^1H_{u_1}(t+r\sigma)\mathrm{d}\sigma+\frac{1}{r}\int_0^t\int_{t-s-r}^{t-s+r}H_u[s](\rho)\mathrm{d}\rho\mathrm{d}s,
\end{align*}
then $u=u(t,r)$ is a solution of the Cauchy problem \eqref{G1}. Note that $v\in X_{\kappa}$ when $\kappa>1$ for the linear Cauchy problem \eqref{Linearized-v} via Proposition \ref{Prop-linear-decay} and initial data $u_0$, $u_1$ with compact support. In order to prove Theorem \ref{Thm-GESDS}, in the subsequent part, we will demonstrate the following two crucial inequalities:
\begin{align}
	\|Lu\|_{X_{\kappa}}&\lesssim \|u\|_{X_{\kappa}}^{p_S(3)},\label{Crucial-01}\\
	\|Lu-L\tilde{u}\|_{X_{\kappa}}&\lesssim \|u-\tilde{u}\|_{X_{\kappa}}\left(\|u\|_{X_{\kappa}}^{p_S(3)-1}+\|\tilde{u}\|_{X_{\kappa}}^{p_S(3)-1} \right),\label{Crucial-02}
\end{align}
for any $u,\tilde{u}\in X_{\kappa}$, under some conditions for the modulus of continuity. Let us recall $u_0=\varepsilon\bar{u}_0$ and $u_1=\varepsilon\bar{u}_1$. Combining \eqref{Crucial-01} with Proposition \ref{Prop-linear-decay}, we immediately claim
\begin{align}\label{Crucial-03}
\|v+Lu\|_{X_{\kappa}}\lesssim\varepsilon\left(\|\bar{u}_0\|_{\ml{A}_{\kappa}}+\|\bar{u}_1\|_{\ml{B}_{\kappa+1}}\right)+ \|u\|_{X_{\kappa}}^{p_S(3)}.
\end{align}
Providing that we take a small parameter $0<\varepsilon<\varepsilon_0\ll1$ and compactly supported initial data, we combine \eqref{Crucial-03} and \eqref{Crucial-02} to claim that there exists a global (in time) small data radial solution $u\in X_{\kappa}$ by using  Banach's fixed point theorem.

From the condition \eqref{GESDS-Condition},  there exists a positive and continuous  function $\bar{\kappa}=\bar{\kappa}(\tau)$ such that
\begin{align}\label{Nee-con}
	\mu(\tau)=\left(\log\frac{1}{\tau}\right)^{-\frac{1}{p_S(3)}}\bar{\kappa}(\tau)\ \ \mbox{and}\ \ \lim\limits_{\tau\to0^+}\bar{\kappa}(\tau)=0.
\end{align}
Furthermore, the condition \eqref{Speical-mu} means that the above function $\bar{\kappa}$ satisfies the next additional condition:
\begin{align}\label{Condition-2}
\bar{\kappa}(\tau)\log\log\frac{1}{\tau}\lesssim 1\ \ \mbox{when}\ \ \tau\in(0,\tau_0].
\end{align}
Before starting our proof, let us introduce a crucial lemma.
\begin{lemma}\label{Lemma-Key-Integral}
	Let us consider a modulus of continuity $\mu=\mu(\tau)$ with $\mu(0)=0$ satisfying the conditions \eqref{GESDS-Condition} and \eqref{Speical-mu}. Let us recall the weighted factor via \eqref{Weighted-function}. Then, the integral
	\begin{align*}
		 I(\xi):=\int_{-|\xi|}^{|\xi|}[\omega(\langle\eta\rangle)]^{-p_S(3)}\langle\xi+\eta\rangle\langle\eta\rangle^{-1}
\mu\left(\varepsilon_0[\omega(\langle\eta\rangle)]^{-1}\langle\xi\rangle^{-1}\langle\eta\rangle^{-\frac{1}{p_S(3)}}\right)\mathrm{d}\eta
	\end{align*}
	fulfills the following estimate:
	\begin{align*}
		I(\xi)\lesssim \langle\xi\rangle [\log(\langle\xi\rangle)]^{-\frac{1}{p_S(3)}}(\log\log\langle\xi\rangle)\,\bar{\kappa}\left(\varepsilon_0\langle\xi\rangle^{-1}\right).
	\end{align*}
\end{lemma}
\begin{proof}
To begin with, let us  split $I(\xi)$ into two parts
	\begin{align*}
		 I_1(\xi)&:=\int_{-\frac{|\xi|}{2}}^{\frac{|\xi|}{2}}[\omega(\langle\eta\rangle)]^{-p_S(3)}\langle\xi+\eta\rangle\langle\eta\rangle^{-1}\mu\left(\varepsilon_0[\omega(\langle\eta\rangle)]^{-1}\langle\xi\rangle^{-1}\langle\eta\rangle^{-\frac{1}{p_S(3)}}\right)\mathrm{d}\eta,\\
		 I_2(\xi)&:=\left(\int_{\frac{|\xi|}{2}}^{|\xi|}+\int_{-|\xi|}^{-\frac{|\xi|}{2}}\right)[\omega(\langle\eta\rangle)]^{-p_S(3)}\langle\xi+\eta\rangle\langle\eta\rangle^{-1}\mu\left(\varepsilon_0[\omega(\langle\eta\rangle)]^{-1}\langle\xi\rangle^{-1}\langle\eta\rangle^{-\frac{1}{p_S(3)}}\right)\mathrm{d}\eta.
	\end{align*}
	For the first integral, we use  the asymptotic behaviors $\langle\xi\rangle\approx\langle\eta+\xi\rangle\approx \langle \eta-\xi\rangle$ when $\eta\in[-\frac{|\xi|}{2},\frac{|\xi|}{2}]$ to deduce
	\begin{align}\label{Est-I1}
		 I_1(\xi)&\lesssim \langle\xi \rangle \int_{-\frac{|\xi|}{2}}^{\frac{|\xi|}{2}}\langle\eta\rangle^{-1}[\log(\langle\eta\rangle)]^{-1}\mu\left(\varepsilon_0\langle\xi\rangle^{-1}\langle\eta\rangle^{-\frac{1}{p_S(3)}}[\log(\langle\eta\rangle)]^{-\frac{1}{p_S(3)}}\right)\mathrm{d}\eta\notag\\
		 &\lesssim \langle\xi \rangle\mu\left(\varepsilon_0\langle\xi\rangle^{-1}\right) \int_{0}^{\frac{|\xi|}{2}}\langle\eta\rangle^{-1}[\log(\langle\eta\rangle)]^{-1}\mathrm{d}\eta\notag\\
%		 &\lesssim\langle\xi\rangle\left[\log\left(\varepsilon_0^{-1}\langle\xi\rangle\right)\right]^{-\frac{1}{p_S(3)}}\bar{\kappa}\left(\varepsilon_0\langle\xi\rangle^{-1}\right)\int_0^{\frac{|\xi|}{2}}\langle\eta\rangle^{-1}[\log(\langle\eta\rangle)]^{-1}\mathrm{d}\eta\notag\\
		&\lesssim \langle\xi\rangle [\log(\langle\xi\rangle)]^{-\frac{1}{p_S(3)}}(\log\log\langle\xi\rangle)\,\bar{\kappa}\left(\varepsilon_0\langle\xi\rangle^{-1}\right),
	\end{align}
	where we used the condition \eqref{Nee-con} and the increasing property with $[\langle\eta\rangle\log(\langle\eta\rangle)]^{-\frac{1}{p_S(3)}}<1$.
	
	Let us turn to the first part of $I_2(\xi)$, which is denoted by $I_{2,1}(\xi)$. When $\xi\geqslant 0$, due to $\eta\in[\frac{\xi}{2},\xi]$, the equivalences $\langle\eta+\xi\rangle\approx\langle \eta\rangle\approx\langle\xi\rangle$ hold. It leads to
	\begin{align*}
		I_{2,1}(\xi)&\lesssim\langle\xi\rangle [\log(\langle\xi\rangle)]^{-1}\mu\left(\varepsilon_0\langle\xi\rangle^{-1-\frac{1}{p_S(3)}}[\log(\langle\xi\rangle)]^{-\frac{1}{p_S(3)}}\right)\\
		&\lesssim \langle\xi\rangle [\log(\langle\xi\rangle)]^{-1-\frac{1}{p_S(3)}}\,\bar{\kappa}\left(\varepsilon_0\langle\xi\rangle^{-1}\right),
	\end{align*}
because $[\langle\xi\rangle\log(\langle\xi\rangle)]^{-\frac{1}{p_S(3)}}<1$.
When $\xi\leqslant0$, the equivalences $\langle\eta-\xi\rangle\approx\langle\xi\rangle\approx\langle\eta\rangle$ are valid for $\eta\in[-\xi,-\frac{\xi}{2}]$. Then, one deduces 
	\begin{align*}
		 I_{2,1}(\xi)&\lesssim[\log(\langle\xi\rangle)]^{-1}\langle\xi\rangle^{-1}\mu\left(\varepsilon_0\langle\xi\rangle^{-1-\frac{1}{p_S(3)}}[\log(\langle\xi\rangle)]^{-\frac{1}{p_S(3)}}\right)\int^{-\frac{\xi}{2}}_{-\xi}\langle -\xi-\eta\rangle\mathrm{d}\eta\\
		&\lesssim \langle\xi\rangle [\log(\langle\xi\rangle)]^{-1-\frac{1}{p_S(3)}}\,\bar{\kappa}\left(\varepsilon_0\langle\xi\rangle^{-1}\right).
	\end{align*}
	The above two estimates are stronger than the estimate \eqref{Est-I1}. Repeating the same procedure as those for $I_{2,1}(\xi)$, by symmetry we are able to get
	\begin{align*}
		I_{2,2}(\xi)\lesssim \langle\xi\rangle [\log(\langle\xi\rangle)]^{-1-\frac{1}{p_S(3)}}\,\bar{\kappa}\left(\varepsilon_0\langle\xi\rangle^{-1}\right).
	\end{align*}
	Summarizing the above derived estimates, we complete the proof of this lemma.
\end{proof}

\subsection{Some estimates for solutions to nonlinear models: Proof of Theorem \ref{Thm-GESDS}}\label{Sub-4.3}
$\ \ \ \ $Let us consider $u\in X_{\kappa}$. From the definition of the Banach space $X_{\kappa}$ with $\kappa>1$, we obtain
\begin{align*}
	|u(s,\rho)|^{p_S(3)}\mu\big(|u(s,\rho)|\big)&\lesssim[\omega(\langle s-|\rho|\rangle)]^{-p_S(3)}\langle s+|\rho|\rangle^{-p_S(3)}\langle s-|\rho|\rangle^{-(\kappa-1)p_S(3)}\|u\|_{X_{\kappa}}^{p_S(3)}\\
	&\quad\times \mu\left([\omega(\langle s-|\rho|\rangle)]^{-1}\langle s+|\rho|\rangle^{-1}\langle s-|\rho|\rangle^{-(\kappa-1)}\|u\|_{X_{\kappa}}\right)\\
	&\lesssim[\omega(\langle s-|\rho|\rangle)]^{-p_S(3)}\langle s+|\rho|\rangle^{-p_S(3)}\langle s-|\rho|\rangle^{-(\kappa-1)p_S(3)}\|u\|_{X_{\kappa}}^{p_S(3)}\\
	&\quad\times \mu\left(\varepsilon_0[\omega(\langle s-|\rho|\rangle)]^{-1}\langle s+|\rho|\rangle^{-1}\langle s-|\rho|\rangle^{-(\kappa-1)}\right),
\end{align*}
where we assumed $\|u\|_{X_{\kappa}}\leqslant \varepsilon_0$ for some $\varepsilon_0>0$ sufficiently small.

With the aim of proving \eqref{Crucial-01}, we just need to show
\begin{align*}
	\omega(\langle t-|r|\rangle)\langle t+|r|\rangle\langle t-|r|\rangle^{\kappa-1}|Lu(t,r)|&\lesssim \|u\|_{X_{\kappa}}^{p_S(3)}.
\end{align*}
Because $Lu$ is even in $r$, we may restrict ourselves to non-negative values of $r$.
Concerning $r\geqslant0$, applying the definition of $Lu$, we get
\begin{align*}
	|Lu(t,r)|&\leqslant\frac{1}{r}\int_0^t\int_{t-s-r}^{t-s+r}|H_u[s](\rho)|\mathrm{d}\rho\mathrm{d}s\\
	&\lesssim\frac{1}{r}\int_0^t\int_{t-s-r}^{t-s+r}\langle\rho\rangle[\omega(\langle s-|\rho|\rangle)]^{-p_S(3)}\langle s+|\rho|\rangle^{-p_S(3)}\langle s-|\rho|\rangle^{-(\kappa-1)p_S(3)}\\
	&\quad\quad\underbrace{\qquad\ \ \qquad\times \mu\left(\varepsilon_0[\omega(\langle s-|\rho|\rangle)]^{-1}\langle s+|\rho|\rangle^{-1}\langle s-|\rho|\rangle^{-(\kappa-1)}\right)\mathrm{d}\rho\mathrm{d}s}_{=:I_0(t,r)}\|u\|_{X_{\kappa}}^{p_S(3)}
\end{align*}
for the case $t\geqslant r$.  In the case $t\leqslant r$, we can slightly modify the representation formulate for $Lu$. Precisely, being $H_u[s](\rho)$ an odd function with respect to $\rho$, one notices
\begin{align*}
	\int_{(t-s)-r}^{r-(t-s)}{H}_{u}[s](\rho)\mathrm{d}\rho=0.
\end{align*}
 The additivity of the integral regions shows
\begin{align*}
	Lu(t,r)=\frac{1}{r}\int_0^t\int_{r-(t-s)}^{(t-s)+r}{H}_u[s](\rho)\mathrm{d}\rho\mathrm{d}s.
\end{align*}
As a consequence, when $t\leqslant r$, we may replace $I_0(t,r)$ by
\begin{align*}
	\widetilde{I}_0(t,r)&:=\int_0^t\int_{r-(t-s)}^{r+(t-s)}\langle \rho\rangle[\omega(\langle s-|\rho|\rangle)]^{-p_S(3)}\langle s+|\rho|\rangle^{-p_S(3)}\langle s-|\rho|\rangle^{-(\kappa-1)p_S(3)}\\
	& \qquad\qquad\qquad\ \ \times \mu\left(\varepsilon_0[\omega(\langle s-|\rho|\rangle)]^{-1}\langle s+|\rho|\rangle^{-1}\langle s-|\rho|\rangle^{-(\kappa-1)}\right)\mathrm{d}\rho\mathrm{d}s.
\end{align*}
All in all, we already derived
\begin{align*}
|Lu(t,r)|\lesssim\begin{cases}
r^{-1}I_0(t,r)\|u\|_{X_{\kappa}}^{p_S(3)}&\mbox{when}\ \ t\geqslant r,\\[0.2em]
r^{-1}\widetilde{I}_0(t,r)\|u\|_{X_{\kappa}}^{p_S(3)}&\mbox{when}\ \ t\leqslant r,
\end{cases}
\end{align*}
for any $r\geqslant0$.
\medskip

\noindent\textbf{Estimates for $I_0(t,r)$ with $t\geqslant r$.} Since $|t-s-r|\leqslant t-s+r$ (we are working with $r\geqslant0$)  and the even function with respect to $\rho$ of the integrand in $I_0(t,r)$ so that
\begin{align*}
	I_0(t,r)&\leqslant 2\int_0^t\int_{\max\{0,t-s-r\}}^{t-s+r}\langle s+\rho\rangle^{-p_S(3)}\langle s-\rho\rangle^{-(\kappa-1)p_S(3)}\langle\rho\rangle[\omega(\langle s-\rho\rangle)]^{-p_S(3)}\\
	&\qquad\qquad\qquad\quad\qquad\times \mu\left(\varepsilon_0[\omega(\langle s-\rho\rangle)]^{-1}\langle s+\rho\rangle^{-1}\langle s-\rho\rangle^{-(\kappa-1)}\right)\mathrm{d}\rho\mathrm{d}s.
\end{align*}
We now perform the change of two variables
\begin{align}\label{Change_of-variable}
	\xi=s+\rho,\ \ \eta=\rho-s, \ \ \mbox{namely,}\ \ \rho=\frac{\xi+\eta}{2},\ \ s=\frac{\xi-\eta}{2}.
\end{align}
Since $\rho,s\geqslant0$, we get $|\eta|\leqslant \xi$. Furthermore, $\rho\in[\max\{0,t-s-r\},t-s+r]$ leads to
\begin{align*}
	t-r\leqslant s+\max\{0,t-s-r\}\leqslant s+\rho= \xi\leqslant t+r.
\end{align*}
Hence, recalling $\kappa:=1+\frac{1}{p_S(3)}$ in the definition of the Banach space $X_{\kappa}$, we may employ Lemma \ref{Lemma-Key-Integral} to derive
\begin{align*}
	 I_0(t,r)&\lesssim\int_{t-r}^{t+r}\langle\xi\rangle^{-p_S(3)}\int_{-\xi}^{\xi}[\omega(\langle\eta\rangle)]^{-p_S(3)}\langle\xi+\eta\rangle\langle\eta\rangle^{-1}\mu\left(\varepsilon_0[\omega(\langle\eta\rangle)]^{-1}\langle\xi\rangle^{-1}\langle\eta\rangle^{-\frac{1}{p_S(3)}}\right)\mathrm{d}\eta\mathrm{d}\xi\\
	&\lesssim \int_{t-r}^{t+r}\langle\xi\rangle^{-p_S(3)}I(\xi)\mathrm{d}\xi\\
	& \lesssim \int_{t-r}^{t+r}\langle\xi\rangle^{1-p_S(3)} [\log(\langle\xi\rangle)]^{-\frac{1}{p_S(3)}}(\log\log\langle\xi\rangle)\,\bar{\kappa}\left(\varepsilon_0\langle\xi\rangle^{-1}\right)\mathrm{d}\xi.
\end{align*}
 Thus, from the even behavior of $Lu$ with respect to $r$, taking account of $t\geqslant r\geqslant0$, we may arrive at
\begin{align*}
\omega(\langle t-r\rangle)\langle t+r\rangle\langle t-r\rangle^{\frac{1}{p_S(3)}}|Lu(t,r)|&\lesssim \omega(\langle t-r\rangle)\langle t+r\rangle\langle t-r\rangle^{\frac{1}{p_S(3)}}r^{-1} |I_0(t,r)|\,\|u\|_{X_{\kappa}}^{p_S(3)}\\
&\lesssim J_0(t,r) \|u\|_{X_{\kappa}}^{p_S(3)},
\end{align*}
 where
 \begin{align*}
 	J_0(t,r)&:=[\log(\langle t-r\rangle)]^{\frac{1}{p_S(3)}}\frac{\langle t+r\rangle\langle t-r\rangle^{\frac{1}{p_S(3)}}}{r}\\
 	&\ \ \quad\times \int_{t-r}^{t+r}\langle\xi\rangle^{1-p_S(3)} [\log(\langle\xi\rangle)]^{-\frac{1}{p_S(3)}}(\log\log\langle\xi\rangle)\,\bar{\kappa}\left(\varepsilon_0\langle\xi\rangle^{-1}\right)\mathrm{d}\xi.
 \end{align*}
Noticing that $\varepsilon_0\langle\xi\rangle^{-1}\leqslant \varepsilon_03^{-1}\leqslant \tau_0$ with $0<\varepsilon_0\ll 1$, we directly apply \eqref{Condition-2} to derive
\begin{align}\label{tool-02}
\bar{\kappa}\left(\varepsilon_0\langle\xi\rangle^{-1}\right)\left[\log\log\left(\varepsilon_0^{-1}\langle\xi\rangle\right)\right]\lesssim 1.
\end{align}
Due to $\varepsilon_0^{-1}\gg1$, we simplify our aim as
\begin{align*}
	J_0(t,r)\lesssim [\log(\langle t-r\rangle)]^{\frac{1}{p_S(3)}}\frac{\langle t+r\rangle\langle t-r\rangle^{\frac{1}{p_S(3)}}}{r} \int_{t-r}^{t+r}\langle\xi\rangle^{1-p_S(3)} [\log(\langle\xi\rangle)]^{-\frac{1}{p_S(3)}}\mathrm{d}\xi.
\end{align*}
Next, we will estimate $J_0(t,r)$ precisely in different zones of the $(t,r)$-plane.
\begin{description}
	\item[\emph{Zone I}: $t\geqslant 2r\geqslant 0$.]  For $\xi\in[t-r,t+r]$, we know the equivalence $\langle\xi\rangle\approx\langle t+r\rangle$. It follows 
	\begin{align*}
	J_0(t,r)&\lesssim\langle t+r\rangle^{2-p_S(3)}\langle t-r\rangle^{\frac{1}{p_S(3)}} [\log(\langle t+r\rangle)]^{-\frac{1}{p_S(3)}}[\log(\langle t-r\rangle)]^{\frac{1}{p_S(3)}}\\
	&\lesssim \langle t+r\rangle^{2-p_S(3)+\frac{1}{p_S(3)}}\lesssim 1,
	\end{align*}
since $-p_S^2(3)+2p_S(3)+1=0$ for three dimensions (see \eqref{Qud-Strauss} with $n=3$).
\item[\emph{Zone II}: $0\leqslant r\leqslant 1$ and $t\leqslant 2r$.] We own the equivalence $\langle t+r\rangle\approx 3$ so that
\begin{align*}
	J_0(t,r)\lesssim \langle t+r\rangle^{1+\frac{1}{p_S(3)}}\langle t-r\rangle^{1-p_S(3)}\lesssim 1
\end{align*}
by using the bounds $2\leqslant \langle t-r\rangle\leqslant\langle t+r\rangle$.
\item[\emph{Zone III}: $r\geqslant 1$ and $r\leqslant t\leqslant 2r$.] Via the equivalences $r\approx\langle r\rangle\approx \langle t+r\rangle$, we are able to conclude
\begin{align*}
J_0(t,r)\lesssim\frac{\langle t+r\rangle}{r}\langle t-r\rangle^{\frac{1}{p_S(3)}}\int_{t-r}^{t+r}\langle\xi\rangle^{1-p_S(3)}\mathrm{d}\xi\lesssim \langle t-r\rangle^{2-p_S(3)+\frac{1}{p_S(3)}}\lesssim 1,
\end{align*}
where we applied $p_S(3)=1+\sqrt{2}>2$.
\end{description}
Summarizing, we have derived the key estimate
\begin{align*}
\omega(\langle t-r\rangle)\langle t+r\rangle\langle t-r\rangle^{\frac{1}{p_S(3)}}|Lu(t,r)|\lesssim J_0(t,r)\|u\|_{X_{\kappa}}^{p_S(3)}\lesssim \|u\|_{X_{\kappa}}^{p_S(3)}
\end{align*}
for any $t\geqslant r$.\medskip

\noindent\textbf{Estimates for $\widetilde{I}_0(t,r)$ with $t\leqslant r$.} This part just discusses the case for $r\geqslant 1$ since the situation in a compact set $t\leqslant r\leqslant 1$ is trivial. As before, we take the change of variables \eqref{Change_of-variable}. According to $s\geqslant0$ and $\rho\geqslant r-(t-s)$, we deduce $r-t\leqslant \eta\leqslant\xi$, while from $s\geqslant 0$ and $|\rho-r|\leqslant t-s$ it follows $r-t\leqslant \xi\leqslant r+t$. For this reason, we find
\begin{align*}
	 \widetilde{I}_0(t,r)\lesssim\int_{r-t}^{r+t}\langle\xi\rangle^{-p_S(3)}\int_{r-t}^{\xi}[\omega(\langle\eta\rangle)]^{-p_S(3)}\langle\xi+\eta\rangle\langle\eta\rangle^{-1}\mu\left(\varepsilon_0[\omega(\langle\eta\rangle)]^{-1}\langle\xi\rangle^{-1}\langle\eta\rangle^{-\frac{1}{p_S(3)}}\right)\mathrm{d}\eta\mathrm{d}\xi.
\end{align*}
The relation $[r-t,\xi]\subset[-\xi,\xi]$ and the application of Lemma \ref{Lemma-Key-Integral} associated with \eqref{tool-02} yield
\begin{align*}
	\widetilde{I}_0(t,r)&\lesssim\int_{r-t}^{r+t}\langle\xi\rangle^{1-p_S(3)} [\log(\langle\xi\rangle)]^{-\frac{1}{p_S(3)}}(\log\log\langle\xi\rangle)\,\bar{\kappa}\left(\varepsilon_0\langle\xi\rangle^{-1}\right)\mathrm{d}\xi\\
	&\lesssim[\log(\langle r-t\rangle)]^{-\frac{1}{p_S(3)}}\int_{r-t}^{r+t}\langle\xi\rangle^{1-p_S(3)} \mathrm{d}\xi\\
	&\lesssim \langle t-r\rangle^{2-p_S(3)}[\log(\langle t-r\rangle)]^{-\frac{1}{p_S(3)}},
\end{align*}
due to $2-p_S(3)=1-\sqrt{2}<0$. Hence,
\begin{align*}
\omega(\langle t-r\rangle)\langle t+r\rangle\langle t-r\rangle^{\frac{1}{p_S(3)}}|Lu(t,r)|&\lesssim\omega(\langle t-r\rangle)\langle t+r\rangle\langle t-r\rangle^{\frac{1}{p_S(3)}}r^{-1}|\widetilde{I}_0(t,r)|\,\|u\|_{X_{\kappa}}^{p_S(3)}\\
& \lesssim\frac{\langle t+r\rangle}{r}\langle t-r\rangle^{2-p_S(3)+\frac{1}{p_S(3)}}\|u\|_{X_{\kappa}}^{p_S(3)}\\
&\lesssim \|u\|_{X_{\kappa}}^{p_S(3)}
\end{align*}
by the same reason as the one in \emph{Zone III} of estimates for $I_0(t,r)$.

Thus, we may claim
\begin{align}\label{Chen-Ineq-01}
\big\|\omega(\langle t-|r|\rangle)\langle t+|r|\rangle\langle t-|r|\rangle^{\frac{1}{p_S(3)}}Lu(t,r)\big\|_{L^{\infty}([0,+\infty)\times\mb{R})}\lesssim \|u\|_{X_{\kappa}}^{p_S(3)}
\end{align}
with the weighted factor defined in \eqref{Weighted-function} and the conditions \eqref{GESDS-Condition} as well as \eqref{Speical-mu}. As a by-product, due to
\begin{align*}
	|Lu(t,r)-L\tilde{u}(t,r)|\lesssim\frac{1}{r}\|u-\tilde{u}\|_{X_{\kappa}}\left(\|u\|_{X_{\kappa}}^{p_S(3)-1}+\|\tilde{u}\|_{X_{\kappa}}^{p_S(3)-1} \right)\times
	\begin{cases}
		r^{-1}I_0(t,r)&\mbox{when}\ \ t\geqslant r,\\
		r^{-1}\widetilde{I}_0(t,r)&\mbox{when}\ \ t\leqslant r,
	\end{cases}
\end{align*}
one can prove
\begin{align*}
	&\big\|\omega(\langle t-|r|\rangle)\langle t+|r|\rangle\langle t-|r|\rangle^{\frac{1}{p_S(3)}}\big(Lu(t,r)-L\tilde{u}(t,r)\big)\big\|_{L^{\infty}([0,+\infty)\times\mb{R})}\\
	&\qquad\lesssim \|u-\tilde{u}\|_{X_{\kappa}}\left(\|u\|_{X_{\kappa}}^{p_S(3)-1}+\|\tilde{u}\|_{X_{\kappa}}^{p_S(3)-1} \right),
\end{align*}
by the same way as those for \eqref{Chen-Ineq-01}.

Summarizing the last statements, we have completed the proof of the desired inequalities \eqref{Crucial-01} as well as \eqref{Crucial-02}, namely, the proof of Theorem \ref{Thm-GESDS} is completed. Furthermore, by employing Banach's fixed point theorem, with the weighted data, we can get the next global (in time) existence result with some pointwise decay estimates.
\begin{coro}\label{Coro-Decay}
	Let us consider a modulus of continuity $\mu=\mu(\tau)$ with $\mu(0)=0$ fulfilling \eqref{GESDS-Condition} and \eqref{Speical-mu}. Let
	 $u_0\in(\ml{A}_{\kappa}\cap \ml{C}^2)$ and $u_1\in(\ml{B}_{\kappa+1}\cap \ml{C}^1)$  be radial with $\kappa=1+\frac{1}{p_S(3)}$. Then, there exists $0<\varepsilon_0\ll 1$ such that for any $\varepsilon\in(0,\varepsilon_0)$ fulfilling  $\|u_0\|_{\ml{A}_{\kappa}}+\|u_1\|_{\ml{B}_{\kappa+1}}<\varepsilon_0$,  the semilinear Cauchy problem \eqref{Semilinear-Wave-Modulus} for $n=3$ admits a uniquely determined global (in time) small data radial solution $u\in\ml{C}([0,+\infty)\times\mb{R}^3)$. Furthermore, the solution fulfills the following pointwise decay estimates:
	   \begin{align*}
	   	|u(t,r)|\lesssim [\log(\langle t-|r|\rangle)]^{-\frac{1}{p_S(3)}}\langle t+|r|\rangle^{-1}\langle t-|r|\rangle^{-\frac{1}{p_S(3)}}\left(\|u_0\|_{\ml{A}_{\kappa}}+\|u_1\|_{\ml{B}_{\kappa+1}}\right).
	   \end{align*}
   Here, the data spaces $\ml{A}_{\kappa}$ and $\ml{B}_{\kappa}$ were defined in Proposition \ref{Prop-linear-decay}.
\end{coro}
\begin{remark}
One of our new tools is the introduction of new weighted data spaces with logarithmic factors so that we finally can derive polynomial-logarithmic type decay estimates for the global (in time) radial solutions to the Cauchy problem (\ref{G1}) with modulus of continuity from the previous Corollary \ref{Coro-Decay}.
\end{remark}

\section{Final remarks}\label{Section-fINAL}
\setcounter{equation}{0}
$\ \ \  \ $In this paper, we derived a  blow-up result and a global (in time) existence result for semilinear classical wave equations with a modulus of continuity in the nonlinearity $|u|^{p_S(n)}\mu(|u|)$. Specially, considering a modulus of continuity $\mu=\mu(\tau)$ with $\mu(0)=0$ carrying the form \eqref{Int-example} with $0<\tau_0\ll 1$, we described the critical regularity of nonlinearities $|u|^{p_S(3)}\mu(|u|)$ by the threshold $\gamma=\frac{1}{p_S(3)}$.

According to the general blow-up condition \eqref{Assumption-Blow-up} and the global (in time) existence result in Theorem \ref{Thm-GESDS}, we expect that the general threshold is described by the quantity \eqref{Conjecture-Chen-Reissig}, i.e.
\begin{align}\label{Conjecture}
C_{\mathrm{Str}}=\lim_{\tau \to 0^+}\mu(\tau)\left(\log\frac{1}{\tau}\right)^{\frac{1}{p_S(n)}}.
\end{align}
Let us explain it via three situations concerning the value of $C_{\mathrm{Str}}\geqslant0$.
\begin{itemize}
	\item {\bf Blow-up of solutions when $C_{\mathrm{Str}}=+\infty$}:  Due to the proposed condition \eqref{Assumption-Blow-up} in Theorem \ref{Thm-Blow-up}, the validity of the last conjecture from the blow-up viewpoint is already known when $C_{\mathrm{Str}}=+\infty$.
	\item {\bf Blow-up of solutions when $C_{\mathrm{Str}}\in(0,+\infty)$}: Due to the proposed condition \eqref{Assumption-Blow-up} in Theorem \ref{Thm-Blow-up}, we already know the blow-up phenomenon occurs when $C_{\mathrm{Str}}\in[c_l,+\infty)$ with a suitably large constant $c_l\gg1$. Recently, the blow-up result for semilinear classical wave equations with modulus of continuity in derivative type nonlinearity has been studied by \cite{Chen=2023}. Considering the model $u_{tt}-\Delta u=|u_t|^{p_G(n)}\mu(|u_t|)$ with the Glassey exponent $p_G(n):=\frac{n+1}{n-1}$ for $n\geqslant 2$, the author of \cite{Chen=2023} proposes a new blow-up condition even for the completed intermediate case as follows:
	\begin{align*}
		\lim\limits_{\tau\to 0^+}\mu(\tau)\left(\log\frac{1}{\tau}\right)=C_{\mathrm{Gla}}\in(0,+\infty].
	\end{align*}
 Motivated by the above explanation, we believe the blow-up result still holds in the intermediate case $C_{\mathrm{Str}}\in(0,c_l)$. However, due to some technical difficulties, the rigorous justification is still challenging.
\item {\bf Global (in time) existence of solutions when $C_{\mathrm{Str}}=0$}: Due to the proposed condition \eqref{GESDS-Condition}, i.e. $C_{\mathrm{Str}}=0$,  and the decay assumption \eqref{Speical-mu} in Theorem \ref{Thm-GESDS}, our conjecture  is partially verified from the global (in time) existence perspective.
\end{itemize}
Lastly, we underline that the validity of our conjecture \eqref{Conjecture} has been verified in the present manuscript for the semilinear three dimensional Cauchy problem \eqref{Semilinear-Wave-Modulus} with a modulus of continuity fulfilling $\mu(0)=0$ and $\mu(\tau)=c_l(\log\frac{1}{\tau})^{-\gamma}$ with $c_l\gg1$ when $\tau\in(0,\tau_0]$, because the global (in time) existence result for $\gamma>\frac{1}{p_S(3)}$ and the blow-up result for $0<\gamma\leqslant\frac{1}{p_S(3)}$  have been rigorously demonstrated in Theorems \ref{Thm-GESDS} and \ref{Thm-Blow-up}, respectively.

\section*{Acknowledgments}
The first author was partially supported by the National Natural Science Foundation of China (grant No. 12171317) and Guangdong Basic and Applied Basic Research Foundation (grant No. 2023A1515012044).

\end{document}